\documentclass[a4paper, 10pt, twoside, notitlepage]{amsart}

\usepackage[utf8]{inputenc}
\usepackage{color}
\usepackage{amsmath} 
\usepackage{amssymb} 
\usepackage{amsthm}
\usepackage{geometry}
\usepackage{graphicx}
\usepackage{esint}
\usepackage[ocgcolorlinks,linkcolor=blue]{hyperref}

\numberwithin{equation}{section}
\newtheorem{theorem}{Theorem}[section]

\newtheorem{proposition}[theorem]{Proposition}
\newtheorem{lemma}[theorem]{Lemma}
\theoremstyle{definition}

\newtheorem{remark}[theorem]{Remark}

\newcommand{\R}{\mathbb{R}} 
\newcommand{\C}{\mathbb{C}} 
\newcommand{\N}{\mathbb{N}} 
\newcommand{\E}{\mathbb{E}} 
\newcommand{\PP}{\mathbb{P}} 
\newcommand{\p}{\partial}

\newcommand{\eps}{\varepsilon}
\newcommand{\Hcurl}{H(\mathrm{curl}; \Omega)}
\newcommand{\wt}{\widetilde}
\newcommand{\curl}{\mathrm{curl\,}}
\newcommand{\Div}{\mathrm{Div\,}}
\newcommand{\Curl}{\mathrm{Curl\,}}

\author[P. Caro, R.-Y. Lai, Y.-H. Lin, and T. Zhou]{Pedro Caro \and Ru-Yu Lai \and Yi-Hsuan Lin \and Ting Zhou}
\title[Boundary determination of electromagnetic and Lam\'e parameters]{Boundary determination of electromagnetic and Lam\'e parameters with corrupted data}
\date{}
\address{Basque Center for Applied Mathematics, Bilbao, Spain}
\email{pcaro@bcamath.org}
\address{School of Mathematics, University of Minnesota, USA}
\email{rylai@umn.edu}
\address{Department of Mathematics and Statistics, University of Jyv\"askyl\"a, Finland}
\email{yihsuanlin3@gmail.com}
\address{Department of Mathematics, Northeastern University, USA}
\email{t.zhou@northeastern.edu}

\begin{document}

\begin{abstract} 
We study boundary determination for an inverse problem associated to the time-harmonic Maxwell equations and another associated to the isotropic elasticity system. We identify the electromagnetic parameters and the Lam\'e moduli for these two systems from the corresponding boundary measurements. In a first step we reconstruct Lipschitz magnetic permeability, electric permittivity and conductivity on the surface from the ideal boundary measurements. Then, we study inverse problems for Maxwell equations and the isotropic elasticity system assuming that the data contains measurement errors. 
For both systems, we provide explicit formulas to reconstruct the parameters on the boundary as well as its rate of convergence formula.

	
	\medskip

	
\end{abstract}

\maketitle

\setcounter{tocdepth}{1}

\section{Introduction}

There are several results available, \cite{caro2014global,pichler2018inverse}, for the inverse problem consisting in determining the electromagnetic parameters with low regularity inside a bounded medium with a Lipschitz boundary, using boundary measurements of electromagnetic fields. Typically, the method used in these results already assumes unique determination of the parameters on the boundary of the medium. In this article, we first address such boundary determination of the electromagnetic parameters. We then provide the analysis of the boundary determination of parameters for both Maxwell and elastic systems with corrupted data. In \cite{lin2017boundary}, the boundary determination of the Lam\'e parameters for an isotropic elasticity system has been investigated.

\subsection{Maxwell system}

We first formulate the inverse problem for Maxwell's equations. 
Let $\Omega\subset\R^3$ be a bounded domain with a Lipschitz boundary $\partial\Omega$. Consider real-valued functions $\mu, \eps, \sigma$, first in the space $L^\infty(\Omega)$, representing the magnetic permeability, electric permittivity and electric conductivity, respectively. Furthermore, they satisfy 
\begin{align}\label{elliptic condition}
\mu(x)\geq\mu_0>0, \  \;\eps(x)\geq\eps_0>0 \mbox{ and } \sigma(x) \geq 0,
\end{align}
almost everywhere (a.e.) $x  \in \Omega$, for some positive constants $\mu_0$ and $\eps_0$. 
Suppose that we have access to the boundary measurements of all electromagnetic waves that are time-harmonic with angular frequency $\omega>0$. Then, let $(E,H)$ be an electromagnetic field satisfying 
time-harmonic Maxwell system, either 
\begin{align}\label{Main eq 1}
\begin{cases}
\mathrm{curl }~ E-i\omega\mu H=0 &\text{ in }\Omega,\\
\mathrm{curl }~ H+i\omega \gamma E=0 &\text{ in }\Omega,\\
\nu \times E =f &\text{ on }\p\Omega,
\end{cases}
\end{align}
or
\begin{align}\label{Main eq 2}
\begin{cases}
\mathrm{curl }~ E-i\omega\mu H=0 &\text{ in }\Omega,\\
\mathrm{curl }~ H+i\omega \gamma E=0 &\text{ in }\Omega,\\
\nu \times H =g &\text{ on }\p\Omega,
\end{cases}
\end{align}
where $\gamma:=\epsilon + i\sigma/\omega$.
It is known that \eqref{Main eq 1} and \eqref{Main eq 2} are well-posed except at a discrete set of frequencies. Note that for real parameters (i.e. $\sigma =0$), one needs to consider either the vacuum of eigenvalues for the Maxwell operator or replace the following well-defined boundary maps by the Cauchy data set. For the complex parameters (i.e. $\sigma >0$), there are no real eigenvalues. Throughout this paper, we assume that $\omega>0$ is not an eigenvalue of \eqref{Main eq 1} and \eqref{Main eq 2}. Then the {\em boundary admittance map} $\Lambda_{\mu,\gamma}^{A}$ can be defined by
\[\Lambda_{\mu,\gamma}^{A}(f)=\nu\times H|_{\partial\Omega},\]
where $(E, H)\in \Hcurl\times\Hcurl$ satisfies the boundary value problem \eqref{Main eq 1}.
Here $\nu\in (L^\infty(\partial\Omega))^3$ denotes the unit outer normal vector to $\partial\Omega$ and 
\[\Hcurl=\left\{u\in (L^2(\Omega))^3~|~\mathrm{curl}~u\in (L^2(\Omega))^3\right\}.\]
Similarly, one can define the \emph{boundary impedance map} $\Lambda_{\mu,\gamma}^I$ by 
$$
\Lambda_{\mu,\gamma}^I(g)=\nu \times E|_{\p \Omega},
$$
where $(E,H)\in \Hcurl\times\Hcurl$ satisfies the boundary value problem \eqref{Main eq 2}.
In order to reconstruct $\gamma$ and $\mu$, we need to use the whole boundary information $\Lambda_{\mu,\gamma}^{A}$ and $\Lambda_{\mu,\gamma}^{I}$.

The main result for the ideal data case is the unique boundary identifiability of ${\rm Lip}(\overline\Omega)$-parameters $\mu,\gamma$ at frequency $\omega$ from boundary measurements 
\[\Lambda_{\mu,\gamma}^A,\Lambda_{\mu,\gamma}^I :H^{-1/2}(\Div;\partial\Omega)\rightarrow H^{-1/2}(\Div;\partial\Omega).\]
See \eqref{eqn:HDiv} in Section \ref{sec:2} for the definition of $H^{-1/2}(\Div;\partial\Omega)$.

The following result contains the boundary determination of the electromagnetic parameters without noise.

\begin{theorem}[Boundary identifiability of electromagnetic parameters]\label{thm:main}
	Let $\Omega$ be a bounded domain in $\R^3$, where the boundary $\partial\Omega$ is locally described by the graphs of Lipschitz functions, and $\omega>0$. Assume that two sets of parameters $\mu_j$ and $\gamma_j$ for $j\in \{1,2\}$ belong to ${\rm Lip}(\overline\Omega)$, then we have 
	\begin{itemize}
		
	\item[(1)] \textbf{Unique determination.}
	\[ \Lambda_{\mu_1,\gamma_1}^A=\Lambda_{\mu_2,\gamma_2}^A \mbox{ implies that } \gamma_1= \gamma_2\;\;\mbox{a.e. on }\;\partial\Omega \] 
	and
	\begin{align*}
	\Lambda_{\mu_1,\gamma_1}^I=\Lambda_{\mu_2,\gamma_2}^I \text{ implies that  }\mu_1=  \mu_2 \text{ a.e. on }\p \Omega.
	\end{align*} 
	
	\item[(2)] \textbf{Pointwise reconstruction.} For almost every $P\in \p \Omega$, there exists an explicit sequence of localized boundary data $\{f_N\}_{N=1}^
	\infty$ supported around $P$ such that 
	
	\begin{align}\label{recon form 1}
		\lim_{N\to \infty }\frac{i}{\omega}\int_{\partial\Omega}\left[\Lambda^A_{\mu,\gamma}(f_N|_{\partial\Omega})\times\nu\right]\cdot \overline{f_N}~dS= {\gamma}(P)
	\end{align}
	and 
\begin{align}\label{recon form 2}
   	\lim_{N\to \infty }\frac{i}{\omega}\int_{\partial\Omega}\left[\overline{\Lambda^I_{\mu,\gamma}(f_N|_{\partial\Omega})}\right]\cdot (f_N\times\nu)~dS= {\mu}(P).
   \end{align}
    \end{itemize}

\end{theorem} 

\begin{remark}
	In Theorem \ref{thm:main}, the conclusion (2) will imply (1) immediately. Therefore, we only prove the case (2). Note that the boundary data $\{f_N\}_{N=1}^\infty $ stands for electric and magnetic fields on $\p \Omega$ in \eqref{recon form 1} and \eqref{recon form 2}, respectively. 
	
\end{remark}

For the Calder\'on problem, where one aims at determining the conductivity $\sigma$ from the Dirichlet-to-Neumann (DN) map associated to the differential operator $\nabla\cdot(\sigma\nabla \centerdot)$,  the boundary data was first shown in \cite{kohn1984determining} for smooth conductivities, and later generalized in a series of papers \cite{zbMATH00004861,brown2001recovering,garcia2016reconstruction,brown2006identifiability}. In particular, the methods in \cite{brown2001recovering,brown2006identifiability} are constructive. A fundamental insight obtained in \cite{sylvester1988inverse}, is that the DN-map $\Lambda_\sigma$ is a first order pseudo-differential operator whose full symbol carries all information of the conductivity $\sigma$ and its derivatives on the boundary. In the case of systems, the available results in this context are due to Joshi-McDowall \cite{joshi2000total,mcdowall1997boundary}, and Salo-Tzou \cite{salo2009carleman}. 

In our result, since the boundary is Lipschitz, the principal symbol approach 
in \cite{joshi2000total} does not directly apply. We adopt and adapt ideas from \cite{brown2001recovering}, 
which basically removes the need of smoothness ---required to set up the framework of pseudo-differential calculus--- by introducing 
highly oscillatory solutions concentrated near the point of interest.
However, one of the novelties and key ingredients in \cite{brown2001recovering} is the use of Hardy's inequality which seems not to have a clear counterpart in the problem for Maxwell's equations. Thus, we replace this ingredient by a new trick that involves a duality argument. See the proof of Theorem \ref{thm:gamma}.\\

Our next result provides the analysis for reconstructing the values of the parameters on the boundary assuming corrupted boundary measurements. The corruption of the data is usually a result of discretized approximation by real data with errors. A formulation of such measurements was introduced in \cite{caro2017calderon} for the Dirichlet-to-Neumann map in solving the Calder\'on problem, where the random white noise was modeled by a random perturbation in the energy form, that depends on the intensity of the boundary potential and current. To be more specific, we consider
a complete probability space $(\Pi,\mathcal H, \mathbb P)$, 
	and a countable family $\{X_\alpha: \alpha\in \mathbb N^2\}$ of independent complex Gaussian random variables $X_\alpha:\varpi\in\Pi~\mapsto~X_\alpha(\varpi)\in\C$ such that 
	\begin{equation}\label{eqn:GRV}\E X_\alpha=0,\quad \E(X_\alpha\overline{X_\alpha})=1,\quad \E(X_\alpha X_\alpha)=0\qquad \forall \alpha\in \N^2,\end{equation}
	with standard expectation of a random variable defined by 
	\[\E X=\int_\Pi X d\PP.\]
In \cite{caro2017calderon}, the noisy data for the Calder\'on problem is defined as
\[\mathcal N_\sigma(f,g)=\int_{\partial\Omega}\Lambda_\sigma f\overline g ~dS+\sum_{\alpha\in\mathbb N^2}(f| e_{\alpha_1})(g|e_{\alpha_2})X_\alpha\qquad f,g\in H^{1/2}(\partial\Omega), \]
where $\alpha=(\alpha_1,\alpha_2)$ and $\{e_n:n\in\mathbb N\}$ is an orthonormal basis of $L^2(\partial\Omega)$ and $(\phi|\psi)$ denotes the inner product in $L^2(\partial\Omega,\mathbb C)$. Here $\Lambda_\sigma$ denotes the Dirichlet-to-Neumann map from $H^{1/2}(\partial\Omega)$ to $H^{-1/2}(\partial\Omega)$
\[\Lambda_\sigma~:~f~\mapsto~\nu\cdot\sigma\nabla u|_{\partial\Omega},\]
where $u$ is the solution to $\nabla\cdot(\sigma\nabla u)=0$ and $u|_{\partial\Omega}=f$, and $\nu$ is the unit outer normal vector on $\partial\Omega$. It is shown in \cite{caro2017calderon} that at almost every point $P\in\partial\Omega$, with a single realization of $\mathcal N_\sigma$ at explicit oscillatory boundary inputs $f_N$ (such as the traces of \eqref{eqn:vN}) ($N\in\mathbb N$), the boundary value of $\sigma$ at the point $P$ can be recovered almost surely by 
\[\lim_{N\rightarrow\infty}\mathcal N_\sigma(f_N, {f_N})=\sigma(P).\]

Note that the noise introduced in the energy form for the Dirichlet-to-Neumann map above is modeled on $L^2(\partial\Omega)$. In the case of Maxwell's equations, we will see that similar type of noise could be introduced at two different levels: the $H^{-1}(\partial\Omega)$-level which guaranties decay of $\| f_N \|_{(H^{-1}(\partial \Omega))^3}$ in Lipschitz domains, and $ L^2(\partial\Omega) $-level
where there is not decay of $\| f_N \|_{(L^2(\partial \Omega))^3}$ and we need extra regularity for $\partial \Omega$. Starting by defining the corrupted data at the $H^{-1}(\partial\Omega)$-level:
\begin{equation}\label{eqn:N}
\begin{split}\mathcal N^{A}_{\mu,\gamma}(f,g):=&\int_{\partial\Omega}(\Lambda^{A}_{\mu,\gamma}(f)\times\nu)\cdot \overline{g}~dS+\sum_{\alpha\in\N^2}(f|{\mathbf e}_{\alpha_1})(g|{\mathbf e}_{\alpha_2})X_\alpha\\
	\mathcal N^{I}_{\mu,\gamma}(f,g):=&\int_{\partial\Omega}\overline{\Lambda^{I}_{\mu,\gamma}(f)}\cdot (g\times\nu)~dS+\sum_{\alpha\in\N^2}(f|{\mathbf e}_{\alpha_1})(g|{\mathbf e}_{\alpha_2})X_\alpha
	\end{split}
\end{equation}
for $f,g\in H^{-1/2}(\Div;\partial\Omega)\subset (H^{-1}(\partial\Omega))^3$, where $\{\mathbf e_n:n\in\mathbb N\}$ is an orthonormal basis of the Hilbert space $(H^{-1}(\partial\Omega))^3$ and $(\phi|\psi)$ here denotes the inner product in $(H^{-1}(\partial\Omega))^3$. Then we have the following reconstruction formula for the Maxwell system with corrupted data.

\begin{theorem}\label{thm:corrupted}
	Let $\Omega \subset \R^3$ be a bounded Lipschitz domain and $\mu ,\epsilon,\sigma$ be Lipschitz continuous functions satisfying \eqref{elliptic condition}. Let $\mathcal N^{A}_{\mu,\gamma}$ and $\mathcal N^{I}_{\mu,\gamma}$ be the quadratic form given by \eqref{eqn:N}, then for almost every $P\in\partial\Omega$, one has 
	\begin{itemize}
		\item[(1)]  \textbf{Unique determination.} There exists explicit boundary data $\{f_N\}_{N=1}^\infty $ in the space $ H^{-1/2}(\Div;\partial\Omega)$ such that 
		\[\lim_{N\rightarrow\infty}\mathcal N^{A}_{\mu,\gamma}(f_N, {f_N})=\gamma(P),\qquad \lim_{N\rightarrow\infty}\mathcal N^I_{\mu,\gamma}({f_N},f_N)=\mu(P)\] almost surely.
		
		\item[(2)] \textbf{Rates of convergence.} There exist positive constants $C_\gamma $ (depending on $\partial\Omega$ and bounds for $\gamma$) and $C_\mu $ (depending on $\partial\Omega$ and bounds for $\mu$), such that, for every $0<\theta<1$ and $\epsilon>0$, we have
			\[\mathbb{P}\left\{|\mathcal N^A_{\mu,\gamma}(f_N, {f_N})-\gamma(P)|\leq C_\gamma N^{-\theta/ 2}\right\}\geq 1-\epsilon \text{ for any }N\geq c\epsilon^{-{\frac{1}{1-\theta}}},\]
		where the constant $c$ only depends on $C_{\partial\Omega}$ and $\theta$.  A similar estimate holds for $\mu$, that is, 
			\[\mathbb{P}\left\{|\mathcal N^I_{\mu,\gamma}(f_N, {f_N})-\mu(P)|\leq C_\mu N^{-\theta/ 2}\right\}\geq 1-\epsilon \text{ for any }N\geq c\epsilon^{-{\frac{1}{1-\theta}}},\]
		where the constant $c>0$ only depends on $C_{\partial\Omega}$ and $\theta$. 
	\end{itemize}
	
\end{theorem}


\bigskip
Next we consider the problem with error modeled at the $L^2(\partial\Omega)$-level. That is, in the definition \eqref{eqn:N}, we choose $\{\textbf{e}_n:n\in\N\}$ to be an orthonormal basis of $(L^2(\partial\Omega))^3$ with the inner product $(\phi | \psi)=\int_{\partial\Omega}\phi\cdot\overline{\psi}dS$ and $f, g\in (L^2(\partial\Omega))^3$. To make rigorous sense of this definition, we will assume in this discussion that the boundary of the domain is locally defined by the graph of $C^{1,1}$ functions. In this case, the boundary impedance and admittance maps are well-defined for $f,g\in H^{1/2}(\Div,\partial\Omega)$. Unlike the previous case of $(H^{-1}(\partial\Omega))^3$ perturbations, the decaying in $N$ does not hold anymore for $\|f_N\|_{(L^2(\partial\Omega))^3}$. We actually have $\|f_N\|_{(L^2(\partial\Omega))^3}\leq C_{\partial\Omega}$ where $C_{\partial\Omega}$ is a constant depending on the boundary.  This is similar to the reconstruction of the normal derivative of the conductivity with corrupted data in \cite{caro2017calderon}; and similarly, our family of solutions can filter out the noise when averaged with respect to the parameter $N^{1/2}$.  We then obtain the following result.
\begin{theorem}\label{thm:corrupted-L2}
Let $\Omega\subset\R^3$ be a bounded domain whose boundary can be defined by the graphs of $C^{1,1}$-functions, and $\mu,\varepsilon,\sigma$ be Lipschitz continuous functions satisfying \eqref{elliptic condition}. Let $\mathcal N^A_{\mu,\gamma}$ and $\mathcal N^I_{\mu,\gamma}$ be the quadratic form given by \eqref{eqn:N} at the $L^2(\partial \Omega)$-level. Then for every $P\in\partial\Omega$, there exists an explicit family $\{f_t:~t\geq 1\}$ in $H^{1/2}(\Div,\partial\Omega)$ such that  for $N\in \mathbb N\backslash\{0\}$ and $T_N:=N^{3+3\theta/2}$ with $\theta\in (0,1)$,
	\begin{itemize}
		\item[(1)]  \textbf{Unique determination.} 
		\[
		\lim_{N\rightarrow\infty}\frac{1}{T_N}\int_{T_N}^{2T_N}\mathcal N^{A}_{\mu,\gamma}(f_{t^2}, f_{t^2})~dt=\gamma(P),\qquad \lim_{N\rightarrow\infty}\frac{1}{T_N}\int_{T_N}^{2T_N}\mathcal N^I_{\mu,\gamma}(f_{t^2},f_{t^2})~dt=\mu(P)
		\] 
		almost surely.
		
		\item[(2)] \textbf{Rates of convergence.} Set 
		\[Y^A_N=\frac1{T_N}\int_{T_N}^{2T_N}\mathcal N^A_{\mu,\gamma}(f_{t^2}, f_{t^2})~dt,\qquad Y^I_N=\frac1{T_N}\int_{T_N}^{2T_N}\mathcal N^I_{\mu,\gamma}(f_{t^2}, f_{t^2})~dt.\]
		There exist positive constants $C_\gamma>0$ (depending on $\partial\Omega$ and bounds for $\gamma$) and $C_\mu>0$ (depending on $\partial\Omega$ and bounds for $\mu$), such that, for every $0<\theta<1$ and $\epsilon>0$, we have
			\[\mathbb{P}\left\{|Y^A_N-\gamma(P)|\leq C_\gamma N^{-\theta/ 2}\right\}\geq 1-\epsilon \text{ for any }N\geq c_\gamma\epsilon^{-{\frac{1}{1-\theta}}},\]
and
			\[\mathbb{P}\left\{|Y^I_N-\mu(P)|\leq C_\mu N^{-\theta/ 2}\right\}\geq 1-\epsilon \text{ for any }N\geq c_\mu\epsilon^{-{\frac{1}{1-\theta}}},\]
		where the constants $c_\gamma$ and $c_\mu$ depend on $\theta$, $\partial\Omega$, lower bounds for $\varepsilon_0$ and $\mu_0$, and upper bounds for $\|\gamma\|_{\rm Lip(\overline\Omega)}$ and $\|\mu\|_{\rm Lip(\overline\Omega)}$, respectively.
	\end{itemize}	
\end{theorem}

\begin{remark} The reconstruction in Theorem \ref{thm:corrupted} can only be ensured for almost every point at the boundary because of the regularity of the domain. However, the reconstruction formula of Theorem \ref{thm:corrupted-L2} holds for every point since the domain is assumed to have a $C^{1,1}$ boundary.
\end{remark}

If we compare Theorem \ref{thm:corrupted} and Theorem \ref{thm:corrupted-L2} with the results in \cite{caro2017calderon} for the reconstruction of the conductivity and its normal derivative at the boundary, we can see a couple of similarities. When modeled the noise at the $H^{-1}$-level, no averaging is required for the reconstruction, as it happened in \cite{caro2017calderon} for the reconstruction of the conductivity. In \cite{caro2017calderon}, this was a consequence of the rate of concentration of the supports of the family $\{ f_N \}$ around the point to be reconstructed. However, in our Theorem \ref{thm:corrupted} this is due to the regularizing effect of the covariance operator associated to the noise in the $H^{-1}(\partial \Omega)$-level. On the other hand, when modeling the noise at the $L^2(\partial \Omega)$-level, we require to perform an average in the parameter $\sqrt{N}$ (since the radius of the support of $f_N$ shrinks as $1/\sqrt N$) to overcome the lack of decay of $\| f_N \|_{(L^2(\partial \Omega))^3}$. This was exactly the same situation as in \cite{caro2017calderon} for the reconstruction of the normal derivative of the conductivity at the boundary. In these situations, we have to analyze an oscillatory integral, and isolate appropriately the stationary points. These are the contents of Lemma \ref{lem:avg}.
Note that the decaying rate in this lemma suggests that we might still obtain decays in average even if the norms of $f_N$ are increasing as $N$ grows. Consequently, errors modeled in spaces of higher regularities might be potentially filtered.

\subsection{Elasticity system}

For the second system, we consider the boundary determination of the Lam\'e parameters for the isotropic elasticity equations. Let $\Omega \subset \R^3$ be a bounded domain, $\lambda(x)$ and $\mu(x)$ be the Lam\'e parameters satisfying the strong convexity condition 
\begin{align}\label{strong convexity}
\mu(x)>0 \text{ and }3\lambda(x)+2\mu(x)>0 \text{ for all }x\in \overline{\Omega}.
\end{align}	

The boundary value problem for the isotropic elasticity system is given by 
\begin{align}\label{Lame system}
\begin{cases}
(\nabla \cdot (\mathbf{C}\nabla u))_i = \sum_{j,k,l=1}^3\frac{\p}{\p x_j }\left(C_{ijkl}\frac{\p}{\p x_l}u_k\right)=0 \quad (i=1,2,3) &\text{ in }\Omega,\\
u=f & \text{ on }\p \Omega,
\end{cases}
\end{align}
where $u=(u_1,u_2,u_3)$ is the displacement vector, $\mathbf{C}=\left(C_{ijkl}\right)_{1\leq i,j,k,l\leq 3}$ and 
\begin{align}\label{elastic four tensor}
C_{ijkl}=\lambda \delta_{ij}\delta_{kl} + \mu (\delta_{ik}\delta_{jl}+\delta_{il}\delta_{jk})\text{ for }1\leq i,j,k,l\leq 3
\end{align}
is the isotropic elastic four tensor with Kronecker delta $\delta_{ij}$. One can easily see that $C_{ijkl}$ given by \eqref{elastic four tensor} satisfies the major and minor symmetries, i.e., 
$$
C_{ijkl}=C_{klij}=C_{jikl}, \text{ for }1\leq i,j,k,l\leq 3.
$$

The Dirichlet-to-Neumann (DN) map for the isotropic elasticity system is defined by 
\begin{equation}\label{eqn:DN-map_elasticity}
\Lambda_{\mathbf{C}}~:~(H^{1/2}(\p \Omega))^3 \to (H^{-1/2 }(\p \Omega))^3\text{ with }\left(\Lambda_{\mathbf{C}}f\right)_i = \left. \sum_{j,k,l=1}^3 \nu_j C_{ijkl}\dfrac{\p u_k}{\p x_l}\right|_{\p \Omega}
\end{equation}
for $i=1,2,3$, where $u\in (H^1(\Omega))^3 $ is the solution to \eqref{Lame system} and $\nu=(\nu_1,\nu_2,\nu_3)$ is the unit outer normal on $\p \Omega$. The inverse problem is whether the elastic tensor $\mathbf{C}$ is uniquely determined by $\Lambda_{\mathbf{C}}$, and to calculate $\mathbf{C}$ of $\Lambda_{\mathbf{C}}$ if $\mathbf{C}$ is determined by $\Lambda_{\mathbf{C}}$. Note that the global uniqueness for the isotropic elasticity system stays open for the three-dimensional case and it was solved in \cite{imanuvilov2015global} for the two-dimensional case.

The boundary determination of the zeroth order and higher order Lam\'e moduli was studied by \cite{tanuma2007stroh} and \cite{lin2017boundary}, respectively. In other words, given any $P\in \p \Omega$ (when $\p \Omega$ and the Lam\'e moduli are sufficiently smooth), one can derive reconstruction formulas for the Lam\'e moduli $\lambda$ and $\mu$ and their derivatives at $P\in \p \Omega$, from the localized DN map. Now, our goal is to give a similar reconstruction algorithm for the Lam\'e parameters with corrupted data. 

Due to the existence of elliptic regularity theory for this system, the corrupted data for the elastic system is similar to that of the scalar conductivity equation discussed in \cite{caro2017calderon}, namely, the random noise is introduced at $(L^2(\partial\Omega))^3$ vector level by introducing the bilinear form with corrupted data 
\[
\mathcal{N}_{\mathbf{C}}(f,g):=\int_{\p \Omega}\Lambda_{\mathbf{C}}f\cdot \overline{g}\ dS+\sum _{\alpha \in \N^2}(f|\mathbf e_{\alpha_1})(g|\mathbf e_{\alpha_2})X_\alpha,
\]
for $f, g\in (H^{1/2}(\partial\Omega))^3$, where $\{\mathbf e_n:n\in\mathbb N\}$ is an orthonormal basis of the Hilbert space $(L^2(\partial\Omega))^3$ and $(\phi|\psi)$ here denotes the inner product in $(L^2(\partial\Omega))^3$.  Then our results for the elasticity system is as follows:

\begin{theorem}\label{Thm:elastic boundary determination}
	Let $\Omega\subset \R^3$  be a bounded Lipschitz domain. Let $\mathbf{C}$ be a Lipschitz continuous elastic four tensor in $\overline{\Omega}$.  
	Then for almost every $P\in \p \Omega$, one has 
	\begin{itemize}
		\item[(1)] \textbf{Unique determination.} There exists an explicit boundary data $\{f_N\}_{N=1}^\infty $ in $ (H^{1/2}(\p \Omega))^3$ such that 
		$$
		\lim_{N\to \infty} \mathcal{N}_{\mathbf{C}}(f_N, f_N )=Z(P)
		$$
		almost surely, where $Z(P)=(Z_{ij})_{1\leq i,j\leq 3}(P)$ with $Z_{ij}=\overline{Z_{ji}}$ for $1\leq i,j\leq3$, and 
		
		\begin{align}\label{eq:Z_ij}
		\begin{split}
		&Z_{ii}  =  \dfrac{\mu}{\lambda+3\mu}\big(2(\lambda+2\mu)-(\lambda+\mu)\iota_{i}^{2}\big),\\
		&Z_{ij}  =  \dfrac{\mu}{\lambda+3\mu}\big(-(\lambda+\mu)\iota_{i}\iota_{j}+\sqrt{-1}(-1)^{k}2\mu\,\iota_{k}\big),\mbox{ }1\leq i<j\leq3
		\end{split}
		\end{align}
		with $(\iota_{1},\iota_{2},\iota_{3})=(\omega_{2},-\omega_{1},0)$
		and the index $k\in\mathbb{N}$ satisfies the condition $1\leq k\leq3$, $k\neq i,j$.

		\item[(2)] \textbf{Rates of convergence.} There exists a constant $C>0$, independent of $N$, such that, for every $0<\theta<1$ and $\epsilon >0$, we have 
		\begin{align}\label{auxiliary rates conv}
		\mathbb{P}\left\{|\mathcal{N}_{\mathbf{C}}(f_N,f_N )-Z(P)|\leq CN^{-\theta/2}\right\}\geq 1-\epsilon \text{ for any }N\geq c\epsilon^{-\frac{1 }{1-\theta}},
		\end{align}
		where the constant $c>0$ depends only on $C_{\p \Omega}$ and $\theta$. 
	\end{itemize}

\end{theorem}

 Theorem \ref{Thm:elastic boundary determination} shows that when the domain $\Omega$ is Lipschitz and $\mathbf C$ is Lipschitz continuous, then one can reconstruct the Lam\'e moduli at almost every boundary point $P\in\partial\Omega$ in a constructive way.

\subsection{Outline}
The rest of this paper is organized as follows. The reconstruction formulas for Lipschitz parameters $\mu$ and $\gamma$ in Maxwell's equations on a Lipschitz boundary $\partial\Omega$ are given in Section \ref{sec:2}. In Section \ref{sec:3}, we analyze the reconstruction with corrupted data by random white noise for the Maxwell equations. The analysis for the reconstruction of the Lipschitz Lam\'e moduli for the isotropic elasticity system with corrupted data is given in Section \ref{sec:4}.


\section{Boundary determination of electromagnetic parameters}\label{sec:2}

First, let us define several function spaces and notations.
\subsection{Preliminaries}
Let us begin with some definitions of function spaces, where the impedance map is well-defined. For a bounded Lipschitz domain $\Omega$, we adopt Tartar's definition (see \cite{tartar1997characterization} or \cite{buffa2002traces}) of the space
\begin{equation}\label{eqn:HDiv}\begin{split}H^{-1/2}(\Div;\partial\Omega):=\Big\{&u\in(H^{-1/2}(\partial\Omega))^3~|~\exists~ \eta\in H^{-1/2}(\partial\Omega), \;\mbox{s.t.,} \\ &\int_{\partial\Omega}u\cdot\nabla\phi~dS=\int_{\partial\Omega}\eta\phi~dS\;\mbox{for }\;\phi\in H^2(\Omega)\Big\},\end{split}\end{equation}
where $(H^{-1/2}(\partial\Omega))^3 $ is the dual space of $(H^{1/2}(\partial\Omega))^3$. 
This implies in a weak sense that $\eta=-\Div u$, where $\Div$ denotes the surface divergence, and that $\nu\cdot u|_{\partial\Omega}=0$, based on the identity for $u$ smooth
\[-\int_{\partial\Omega}(\Div u)\phi~dS=\int_{\partial\Omega} u\cdot\nabla\phi~dS-\int_{\partial\Omega}(u\cdot\nu)(\nabla\phi\cdot\nu)~dS.\]
We will also define in the same spirit the space for the surface scalar curl
\begin{equation}\label{eqn:HCurl}\begin{split}H^{-1/2}(\Curl;\partial\Omega):=\Big\{ &u\in (H^{-1/2}(\partial\Omega))^3~|~ \exists~\xi\in H^{-1/2}(\partial\Omega),\;\mbox{s.t., }\\
&\int_{\partial\Omega} (\nu\times u)\cdot \nabla\phi~dS=\int_{\partial\Omega}\xi\phi~dS\;\;\mbox{for }\;\phi\in H^2(\Omega)\\
&\mbox{and }\; \int_{\partial\Omega}u\cdot\nabla \psi~dS=0\;\;\mbox{for }\;\psi\in H^2(\Omega)\cap H^1_0(\Omega)
\Big\}.
\end{split}\end{equation}
Note that the first condition implies in the weak sense that $\xi=-\Curl u$, where $\Curl$ denotes the surface scalar curl, and the second condition in the definition implies weakly the tangentiality $\nu\cdot u|_{\partial\Omega}=0$. 

Moreover, $H^{-1/2}(\Curl;\partial\Omega)$ is the dual of $H^{-1/2}(\Div;\partial\Omega)$.
It is then shown in \cite{buffa2002traces,tartar1997characterization} that the tangential trace map 
\[\begin{split}\tau_t: H(\curl;\Omega)&\rightarrow H^{-1/2}(\Div;\partial\Omega)\\
u&\mapsto\nu\times u|_{\partial\Omega}
\end{split}\]
and the projection map
\[\begin{split}\pi_t: H(\curl;\Omega)&\rightarrow H^{-1/2}(\Curl;\partial\Omega)\\
u&\mapsto(\nu\times u|_{\partial\Omega})\times \nu
\end{split}\]
are both surjective.

In order to reconstruct the values of the parameters, we begin with the following energy identity, which is obtained by integration by parts
\begin{equation}\label{eqn:IBP} 
\frac i\omega\int_{\partial\Omega}(\nu\times ( \nu \times H))\cdot (\nu\times \overline E)~dS=\int_\Omega \gamma|E|^2-\mu|H|^2~dx
\end{equation}
for the solution $(E, H)\in H(\curl;\Omega)\times H(\curl;\Omega)$ to the Maxwell's equations. Here the boundary integral is the parity of $H^{-1/2}(\Div;\partial\Omega)$ and $H^{-1/2}(\Curl;\partial\Omega)$.

In the following we use $d$ to denote the dimension number so one can trace the dependence of the convergence rate on $d$. In all cases considered in this paper including Maxwell system and elasticity system, $d=3$. We denote by $B(x,r)$ the ball centered at $x$ of radius $r>0$ and adopt the coordinate notation $x=(x',x_d)\in\R^{d-1}\times\R$ in $d$ dimensions. Since we will use some results of Brown \cite{brown2001recovering}, we will follow his notation. 

Given a Lipschitz domain $\Omega\subset\R^d$, for each $P:=(p',p_d)\in\partial\Omega$, we consider a change of variable that flattens the boundary near $P$
\begin{equation}\label{eqn:chgvar}
(z',z_d)=F(x', x_d)=\big(x'+p',x_d+\phi(x'+p')\big),
\end{equation}
where $\phi: \R^{d-1}\rightarrow \R$ is Lipschitz such that 
\[\begin{split}&B(P,\rho)\cap\partial\Omega=B(P,\rho)\cap \{z_d=\phi(z')\}\\
&B(P,\rho)\cap\Omega=B(P,\rho)\cap \{z_d>\phi(z')\}
\end{split}\]
for some $\rho>0$. Let $\wt\Omega=F^{-1}(\Omega)\subset\R^d$ and $\partial\wt\Omega$ be its boundary. There exists a $r>0$ such that 
\[B(0,2r)\cap\{x_d=0\}\subset F^{-1}\big(B(P,\rho)\cap\partial\Omega\big)\subset \partial\wt\Omega.\]
Since we are interested in the coefficients at the point $ P $, we focus on reconstructing $\mu(F(0,0))=\mu(p',\phi(p'))$ and $\gamma(F(0,0))=\gamma(p',\phi(p'))$.

Denote 
\begin{equation}\label{eqn:Jacobi}
M(x):=DF^{-1}(F(x))=\left(\frac{dx_i}{dz_j}\right)_{i,j}(F(x)) .
\end{equation}
By the change of coordinates \eqref{eqn:chgvar}, we have the right hand side of \eqref{eqn:IBP} to be
\begin{align}\label{energy}
I:=\int_\Omega {\gamma}|E|^2-\mu|H|^2~dz=\int_{\wt\Omega}({\wt\gamma} \wt E)\cdot\overline{\wt E}-(\wt\mu \wt H)\cdot \overline{\wt H}~dx,
\end{align} 
where 
\[\wt \mu (x):=\mu(F(x)) M(x)M(x)^t,\quad
\wt \gamma (x):=\gamma(F(x)) M(x)M(x)^t,\]
and
\[\wt E(x):=(M(x)^t)^{-1} E(F(x)),\qquad \wt H(x):=(M(x)^t)^{-1} H(F(x)).\]
Furthermore, the electromagnetic field $(\wt E, \wt H)$ (defined as the pull-back of $(E, H)$ by $F:~\wt\Omega\rightarrow \Omega$) satisfies the Maxwell's equations (in the weak sense)
\begin{equation}\label{eqn:tME}
\mathrm{curl }~ \wt E-i\omega\wt\mu \wt H=0,\;\;\mathrm{curl }~ \wt H+i\omega\wt\gamma \wt E=0\quad\mbox{ in }\;\wt\Omega.
\end{equation}
This last point can be justified by checking that $\mathrm{curl }~ \wt E (x) = M (x) (\mathrm{curl }~ E) (F(x)) $.

We now list a couple of properties of the parameters that are required to apply some results of Brown \cite{brown2001recovering}.
First, let us note that $\mu, \gamma\in {\rm Lip}(\overline\Omega)$  satisfy the hypothesis (H1) in \cite{brown2001recovering}, that is,
\begin{equation}\label{eqn:H1}
|\mu(F(x',x_d)) - \mu(F(x',0))| + |\gamma(F(x',x_d)) - \gamma(F(x',0))| \lesssim |x_d|
\end{equation}
for all $|x'| < 2r$. Regarding the hypothesis H2 in \cite{brown2001recovering}, note that
\begin{align*}
& s^{1-d}\int_{|y'|<s}|\wt\gamma(0,0)-\wt\gamma(y',0)|^2~dy' + s^{1-d}\int_{|y'|<s}|\wt\mu(0,0)-\wt\mu(y',0)|^2~dy'\\
  \lesssim &\ s^2 + s^{1-d} \int_{|y'|<s}|\nabla' \phi(y' + p')-\nabla' \phi( p')|^2~dy',
\end{align*}
where the limit of the last term on the right-hand side vanishes, 
when $s$ goes to zero, for almost every $p'$ by the Lebesgue differentiation theorem. Here we denote $\nabla'\phi:=(\partial_1\phi, \partial_2\phi)^t$.

Our reconstruction method only will work for points $P \in \partial \Omega$ so that
\begin{equation}
\lim_{s \to 0} s^{1-d} \int_{|y'|<s}|\nabla' \phi(y' + p')-\nabla' \phi( p')|^2~dy' = 0
\label{lim:lebesgue_point}
\end{equation}
for the corresponding $\phi$ and $p'$. As pointed out before, for 
almost every point in $P\in \partial \Omega$ its corresponding limit 
in \eqref{lim:lebesgue_point} vanishes.

%
%

\subsection{Reconstruction of $\gamma$}\label{sec. 2.2}
We first give an explicit reconstruction formula of $\gamma$ in an admissible point $P \in \partial \Omega$ from the knowledge of the admittance map $\Lambda_{\mu,\gamma}^A$. 
Recall in \cite{brown2001recovering}, a family of functions with special decaying property is constructed as the input of the Dirichlet-to-Neumann map for $\nabla\cdot\sigma\nabla$ to reconstruct $\sigma$. More specifically, this family was given by
\begin{equation}\label{eqn:vN}
v_N(y)=\eta(N^{1/2}|y'|)\eta(N^{1/2}y_d)e^{N(i\alpha-\vec {e}_d)\cdot y},
\end{equation}
where $\vec {e}_d=(0,\cdots, 0,1)\in \R^d$ and $\eta: \R\rightarrow[0,1]$ is a smooth cutoff function which takes value 1 in $B(0,1/2)$ and 0 outside $B(0,1)$, the vector $\alpha\in\R^d$ can be chosen such that
\begin{equation}\label{eqn:alpha-en}
\begin{split}&|M(0)^t\alpha|= |M(0)^t\vec e_d|,\\ 
&\alpha\cdot M(0) M(0)^t\vec e_d=0.\end{split}
\end{equation}
An explicit choice of $\alpha$ is given in \eqref{term:choices}.

We will make an essential use of the gradient fields $\{\nabla v_N\}_N$. More particularly we will choose $(E, H)$ so that their pull-back $(\wt E, \wt H) = (\nabla v_N + w_1, w_2)$ with $w_1$ and $w_2$ solving
\begin{equation}\label{eqn:MErem}
\begin{cases}\curl w_1-i\omega\wt\mu w_2=0 &\text{ in }\wt \Omega , \\
\curl w_2+i\omega\wt\gamma w_1=-i\omega\wt\gamma\nabla v_N & \text{ in }\;\wt\Omega, \\
\nu\times w_1=0 &\text{ on }\p \wt \Omega.
\end{cases}\end{equation}
Note that $\wt \Omega$ is not necessarily locally described by the graph of Lipschitz functions, so in principle, the theory of well-posedness for \eqref{eqn:MErem} should be revisited. In our particular case, the situation is simpler since $\wt \Omega$ is the pull-back of a domain whose boundary is locally described by the graph of a Lipschitz function. Therefore, it is enough to use map $F$ to obtain $(w_1,w_2)$ in $\wt \Omega$ from the corresponding fields in $\Omega$. We will be solving in $\wt \Omega$ in the rest of the paper, and it will always be justified through the map $F$.

The corresponding energy \eqref{energy} for $(\wt E, \wt H)$ is then given by
\begin{equation}\label{eqn:I-terms}
\begin{split}I=&\int_{\wt\Omega}{\gamma}(F(y)) \nabla \overline{v_N} \cdot MM^t\nabla v_N~dy\\
&+\int_{\wt\Omega}{\gamma}(F(y))\big[2\mathfrak{Re}(\nabla \overline{v_N} \cdot MM^tw_1)+\overline{w_1}\cdot MM^tw_1\big]dy \\
&+\int_{\wt \Omega}\mu(F(y))\overline{w_2}\cdot MM^tw_2~dy.
\end{split}\end{equation}

On the other hand, 
the tangential boundary condition of the electric field is transformed according to
\[\nu \times E (F(x)) = DF(x) \wt \nu \times \wt E (x),\]
where $\wt \nu (x) = DF(x)^t \nu(F(x))$. For $N^{-1/2} < 2r$ the support of $\nabla v_N$ is contained on $\{ x_d = 0 \} \cap B(0, 2r)$, and the tangential boundary condition there becomes
\begin{equation}
\nu \times E (F(x', 0)) = DF(x', 0) \big[ \vec{e}_d \times \nabla v_N (x', 0) \big].
\label{id:tangential_flat-part}
\end{equation}

\bigskip

Since H1 and H2 in \cite[Lemma 1]{brown2001recovering} are satisfied, the first term of $I$ satisfies
\begin{align}\label{eqn:lem1-Br}
&\notag \frac{\int_{\wt\Omega}{\gamma}(F(y)) \nabla \overline{v_N}\cdot MM^t\nabla v_N~dy}{N^{\frac{3-d}{2}}}\\
\rightarrow &\ {\gamma}(p',\phi(p'))(1+|\nabla'\phi(p')|^2)\int_{\R^{d-1}}\eta(|x'|)^2~dx',
\end{align}
as $N\rightarrow\infty$.
 

It turns out that this first term dominates, hence provides the reconstruction of $\gamma(F(0))$ knowing $\phi$ and $\eta$.

\begin{theorem}\label{thm:gamma}
	Suppose $\Omega\subset\R^d$ ($d=3$) is a bounded Lipschitz domain. Let $\mu, \varepsilon, \sigma\in {\rm Lip}(\overline{\Omega})$ satisfy \eqref{elliptic condition}. Let $P\in\partial\Omega$ be an admissible point with $F$ as in \eqref{eqn:chgvar}. We define
	\begin{equation}\label{eqn:bdry-fN}f_N(z):=c_0^{-1/2}
	(M(y))^t(\nu(y)\times \nabla v_N(y))|_{y=F^{-1}(z)},\end{equation}
	where 
	\[c_0=(1+|\nabla'\phi(p')|^2)\int_{\R^{d-1}}\eta(|x'|)^2~dx',\]
	and $M$ and $v_N$ are given by \eqref{eqn:Jacobi} and \eqref{eqn:vN}, respectively. Then 
	\[I(f_N|_{\partial\Omega}):=\frac{i}{\omega}\int_{\partial\Omega}\left[\Lambda^A_{\mu,\gamma}(f_N|_{\partial\Omega})\times\nu\right]\cdot \overline{f_N}~dS\rightarrow {\gamma}(P)\]
	as $N\rightarrow\infty$.
\end{theorem}

\begin{proof}
	To show that the rest two terms in \eqref{eqn:I-terms} are lower order terms, 
	it suffices to show that the $(L^2(\wt\Omega))^3$-norms of $w_1$ and $w_2$ are $o(1)$.
	
	First, we need to consider the dual of the standard regularity estimate for the Maxwell's equations, targeting the $L^2$-norm of the solution.

	Notice that the elliptic condition for the parameters is preserved in the following dual problem:
	Given $(G_1, G_2)\in (L^2(\wt\Omega))^6 $, except for a discrete set of frequencies, there exists a unique solution $(u_1, u_2)\in H(\curl;\wt\Omega)\times H(\curl;\wt\Omega)$ to
	\begin{equation}\label{eqn:MErem-dual}
	\begin{cases}
	\curl u_1+i\omega \wt\mu u_2=G_1 & \text{ in }\wt \Omega, \\
	\curl u_2-i\omega\overline{\wt\gamma} u_1=G_2 & \text{ in }\wt \Omega,\\
	\nu\times u_1=0 & \text{ on }\partial\wt\Omega.
	\end{cases}
	\end{equation}
	Furthermore, we have
	\begin{equation}\label{eqn:reg}\|u_1\|_{H(\curl;\wt\Omega)}+\|u_2\|_{H(\curl;\wt\Omega)}\lesssim \|G_1\|_{(L^2(\wt\Omega))^3}+\|G_2\|_{(L^2(\wt\Omega))^3}.\end{equation}
	Then by integration by parts (duality), we have
	\begin{equation}\label{eqn:dual}\begin{split}\left|\int_{\wt\Omega} w_1\cdot \overline{G_2}+w_2\cdot \overline {G_1}~dy\right|=&\left|\int_{\wt\Omega}(-i\omega\wt\gamma\nabla v_N)\cdot \overline{u_1}~dy\right|\\
	\end{split}.\end{equation}
	It then suffices to show that the right hand side is bounded by $o(1)\|u_1\|_{H(\curl;\wt\Omega)}$ since this would imply, using \eqref{eqn:reg},
	\[\|w_1\|_{(L^2(\wt\Omega))^3}+\|w_2\|_{(L^2(\wt\Omega))^3}\leq o(1).\]
	
It is worth noticing that in \cite{brown2001recovering}, Brown used Hardy's inequality to show a similar estimate 
	\[\|\nabla\cdot\wt\gamma\nabla v_N\|_{H^{-1}(\wt\Omega)}=o(1). 
	\]
The main novelty in our approach is to replace the use of Hardy's 
inequality by a duality argument involving the possibility of writing 
$\wt \gamma(0) \nabla e_N$ as the curl of certain vector field $L_N$.
	

	Start by writing $v_N:=\psi_Ne_N$
	with
	\[\psi_N(y):=\eta(N^{1/2}|y'|)\eta(N^{1/2}y_d),\quad e_N(y)=e^{N(i\alpha-\vec e_d)\cdot y}.\]
	We will estimate the three terms of 
	\begin{equation}\label{eqn:est2}
	\wt\gamma\nabla v_N(y)=\wt\gamma(y)\nabla\psi_N e_N+\big(\wt\gamma(y)-\wt\gamma(0)\big)\psi_N\nabla e_N+\wt\gamma(0)\psi_N \nabla e_N.
	\end{equation}
	
	For the first two terms, we only need to control their $L^2$-norms by duality. 
Then we have
		\begin{equation}\label{eqn:est2-term1}
		\begin{split}
		&\quad \|\wt\gamma\nabla\psi_N e_N\|^2_{(L^2(\wt\Omega))^3}\\
		&\lesssim\|\nabla\psi_N e_N\|^2_{(L^2(\wt\Omega))^3}\\
		&= N^{\frac{2-d}{2}}\int_{\R^d}e^{-2N^{1/2}y_d}(\eta'(|y'|)^2\eta(y_d)^2+\eta(|y'|)^2\eta'(y_d)^2)~dy\\&\lesssim N^{\frac{2-d}{2}}\int_0^1 e^{-2N^{1/2}y_d}+e^{-2N^{1/2}y_d}\eta'(y_d)^2~dy_d\\ &\lesssim N^{\frac{2-d}{2}}\left(N^{-1/2}+O(e^{-N^{1/2}})\right)\\
		&=O(N^{\frac{1-d}{2}})=O(N^{-1}).
		\end{split}\end{equation}
	Similarly, 
	we consider the square of $L^2$-norm of the second term
	\begin{equation}
	\begin{split}
	&\quad N^2\int_{\wt\Omega}\left|\left(\wt\gamma(y)-\wt\gamma(0)\right)(i\alpha-\vec e_d)\right|^2\psi_N^2e^{-2Ny_d}~dy\\
	&\lesssim N^2\int_{B(0, N^{-1/2})\times\R^+}\left|\wt\gamma(y)-\wt\gamma(0)\right|^2e^{-2Ny_d}~dy,
	\end{split} 
	\label{es:second_term}
	\end{equation}
	where $B(0, N^{-1/2})$ denotes the ball in $\R^{d - 1}$ centered at $0$ and radius $N^{-1/2}$. It is convenient to write,
    \begin{align*}
    	&\quad \wt\gamma(y)-\wt\gamma(0) \\
    	  &=  \big(\gamma (F(y)) - \gamma (F(0)) \big) M(y)M(y)^t + \gamma(F(0)) \big( M(y)M(y)^t - M(0)M(0)^t \big).
    \end{align*}
	Thus, the right-had side of \eqref{es:second_term} can be bounded by
	\begin{align}\label{equ:max}
	\begin{split}
	& N^2\int_{B(0,N^{-1/2})\times\R^+} |y'|^2 e^{-2Ny_d} ~dy\\
	&\qquad+N^2 \int_{B(0, N^{-1/2})\times\R^+} |\nabla' \phi(y'+ p')-\nabla' \phi(p')|^2 e^{-2Ny_d}~dy.
	\end{split}
	\end{align} 
%
By the \eqref{lim:lebesgue_point}, we have that the previous sum is $o(1)$. It remains to prove 
	\[\left|\int_{\wt\Omega}-i\omega\wt\gamma(0)\psi_N \nabla e_N\cdot \overline{u_1}~dx\right|\leq o(1)\|u_1\|_{H(\curl;\wt\Omega)}.\]
The idea will be to write $\wt \gamma(0) \nabla e_N$ as the curl of certain vector field $L_N$.
	First, we state the explicit expression of the matrices $M$ and $MM^t$ at $0$:
	\[M (0)=\left(\begin{array}{cc}I_{d-1} & 0 \\-\nabla'\phi(p')^{t} & 1\end{array}\right) ,\quad M(0)M^t(0)=\left(\begin{array}{cc}I_{d-1} & -\nabla'\phi(p') \\-(\nabla'\phi(p'))^{t} & 1+|\nabla'\phi(p')|^2\end{array}\right).\]
	Since $\alpha$ is chosen such that $\beta = M(0)^t(i\alpha-\vec e_d) $ satisfies $\beta \cdot \beta = 0$, 
	we have that $\wt\gamma(0)\nabla e_N$ is divergence free, namely,
	\[\nabla\cdot\left(\wt\gamma(0)\nabla e_N(y)\right)=0.\]
	Therefore, there must exist a vector field $L_N = L_N(y)$ such that 
	\[\nabla\times L_N=\wt\gamma(0)\nabla e_N=N\wt\gamma(0)(i\alpha-\vec e_d)e_N.\]
	Next, look for such an $L_N$. We write an ansatz 
	\[L_N=\gamma(F(0))(a+ib)e_N\] and find $a, b\in\R^d$ satisfying the following  algebraic equations 
	\[\begin{split}&\vec e_d\times a+\alpha\times b=M(0)M(0)^t\vec e_d,\\
	&\alpha\times a-\vec e_d\times b=M(0)M(0)^t\alpha.
	\end{split}\]
	It can be verified that in $\R^3$, the choice 
	\begin{equation}
	a=\alpha=\left(\begin{array}{c}\displaystyle\frac{1+|\nabla'\phi|^2}{|\nabla'\phi|}\nabla'\phi \\    \\|\nabla'\phi|\end{array}\right)(p'),\qquad 
	b=\left(\begin{array}{c}-\displaystyle\frac{1}{|\nabla'\phi|}\partial_2\phi \\ \displaystyle\frac{1}{|\nabla'\phi|}\partial_1\phi \\1\end{array}\right)(p'),
	\label{term:choices}
	\end{equation}
	where $\nabla'\phi:=(\partial_1\phi, \partial_2\phi)^t$, qualifies and satisfies $\eta \cdot\eta =0$ and $\eta\cdot\overline\eta=2(1+|\nabla'\phi(p')|^2)$. 
	
	Finally, 
	\begin{equation}\label{eqn:pass_curl}\begin{split}
	&\int_{\wt\Omega}(\psi_N\wt\gamma(0)\nabla e_N)\cdot \overline{u_1}~dy \\
	=&\int_{\wt\Omega}\psi_N(\nabla\times L_N)\cdot \overline{u_1}~dy\\
	=&\int_{\wt\Omega} L_N\cdot \left(\psi_N\nabla\times \overline{u_1}+\nabla\psi_N\times \overline{u_1}\right)~dy\\
	\lesssim &\left({\|\psi_NL_N\|_{(L^2(\wt\Omega))^3}+\|\nabla\psi_N\cdot L_N\|_{ L^2(\wt\Omega)) }}\right)\|u\|_{H(\curl;\wt\Omega)},
	\end{split}\end{equation}
	where we have used that $\nu \times u_1=0$ on $\p \wt \Omega$.
	It is then easy to verify, similar to that for \eqref{eqn:est2-term1}, {$\|\nabla\psi_N\cdot L_N\|_{ 	L^2(\wt\Omega) }=o(1)$. For the other term,
		\[\begin{split}
		\|\psi_N L_N\|_{(L^2(\wt\Omega))^3}^2&\lesssim \int_{\wt\Omega}\eta(N^{1/2}|y'|)^2\eta(N^{1/2}y_d)^2e^{-2Ny_d}~dy\\
		&=N^{-\frac{d}{2}}\int_{\R^d}\eta(|y'|)^2\eta(y_d)^2e^{-2N^{1/2}y_d}~dy\\
		&=O(N^{\frac{-1-d}{2}}) = O(N^{-2}).
		\end{split}\]
	}
	This completes the proof. 
\end{proof}

\subsection{Reconstruction of $\mu$}

{	In order to reconstruct $\mu$, the idea is to let the magnetic energy, namely $\int_{\Omega}\mu|H|^2~dz$, dominate. By symmetry of the equations, $H$ should be chosen roughly $\nabla v_N$, for example, by equating them at the boundary. 
	From now on, we utilize the impedance map, that is, the map 
	\[\Lambda_{\mu,\gamma}^{I}: \nu\times H|_{\partial\Omega}\mapsto \nu\times E|_{\partial\Omega},\]
	then similarly to the previous section, we define our indicator functional being 
	\begin{equation}\label{eqn:JN}J(f_N|_{\partial\Omega}):=\frac i \omega\int_{\partial\Omega}\left[\overline{\Lambda^I_{\mu,\gamma}(f_N|_{\partial\Omega})}\right]\cdot (f_N\times\nu)~dS,\end{equation}
	where $f_N=\nu\times \nabla v_N$ as before. This implies
	\[\begin{split}
	&\ J(f_N|_{\partial\Omega}) \\
	=&\ \int_\Omega\mu|H|^2 - \gamma|E|^2~dx\\
	=&\ \int_{\wt\Omega}\mu(F(y)) \nabla \overline{v_N}\cdot MM^t\nabla v_N~dy\\
	&+\int_{\wt\Omega}{\mu}(F(y))\big[2\mathfrak{Re}(\nabla \overline{v_N}\cdot MM^tw_2)+\overline{w_2}\cdot MM^tw_2\big]dy\\
	&-\int_{\wt \Omega}\gamma(F(y))\overline{w_1}\cdot MM^tw_1~dy,
	\end{split}\]
	where $(w_1, w_2):=(\wt E, \wt H-\nabla v_N)$ in this section and satisfies 
	\begin{equation}
	\begin{cases}
	\curl w_1-i\omega\wt\mu w_2=i\omega\wt\mu\nabla v_N& \text{ in }\wt \Omega, \\
	\curl w_2+i\omega\wt\gamma w_1=0 & \mbox{ in }\;\wt\Omega, \\
	\nu\times w_2=0 & \text{ on }\partial\wt\Omega.
	\end{cases}\end{equation}
	
	Following the proof of Theorem \ref{thm:gamma}, the equation \eqref{eqn:dual} is replaced by
	\begin{equation}\left|\int_{\wt\Omega}w_1\cdot \overline{G_2}+w_2\cdot\overline{G_1}~dy\right|=\left|\int_{\wt\Omega}(-i\omega\wt\mu\nabla v_N)\cdot \overline{u_2}~dy\right|\end{equation}
	for any $(G_1, G_2)\in (L^2(\wt\Omega))^6$, where $(u_1, u_2)$ is the unique solution to
	\begin{align*}
		\begin{cases}
		\curl u_1+i\omega {\wt\mu} u_2=G_1 &\text{ in }\wt \Omega, \\
		\curl u_2-i\omega\overline{\wt\gamma} u_1=G_2 & \text{ in }\wt \Omega,\\
		\nu\times u_2=0 & \text{ on }\partial\wt\Omega.
		\end{cases}
	\end{align*}
	Then it is left to show similarly 
	\[\left|\int_{\wt\Omega}(-i\omega\wt\mu\nabla v_N)\cdot\overline{u_2}~dy\right|=o(1)\|u_2\|_{H(\curl;\wt\Omega)}.\]
	The proof is the same as in Theorem \ref{thm:gamma}. In particular, the integration by parts in \eqref{eqn:pass_curl} is still valid in this case using the boundary condition $\nu\times u_2|_{\partial\wt\Omega}=0$. 
	
	As a result, we obtain the reconstruction formula for $\mu$.
	\begin{theorem} Suppose that $\Omega$, $\mu$, $\varepsilon$, $\sigma$, $P \in \partial \Omega$ and $f_N$ all satisfy the assumptions in Theorem \ref{thm:gamma}. Then we have
		\[\lim_{N\rightarrow\infty}J(f_N|_{\partial\Omega})= \mu(P), \]
		where $J(f_N|_{\partial\Omega})$ is defined by \eqref{eqn:JN}. 
	\end{theorem}
}

\begin{proof}[Proof of Theorem \ref{thm:main}]
	By using all results in Section \ref{sec:2}, we can prove Theorem \ref{thm:main} immediately.
\end{proof}

{
	\section{Boundary determination of electromagnetic parameters with corrupted data}\label{sec:3}
	
	The main objective of this part is to stably identify boundary values of the unknown electromagnetic coefficients from the boundary measurement corrupted by errors, modeled and handled similarly to that in \cite{caro2017calderon} for the Calder\'on problem. 
	
	First, we give a description of the modeling for the random white noise,  first introduced in \cite{caro2017calderon} for the Calder\'on problem, with modifications adopted to the system of Maxwell's equations with our electromagnetic  boundary maps. In particular, the random white noise is introduced to the boundary data on the $H^{-1}(\partial\Omega)$-level as well as on the $L^2(\partial\Omega)$ one. 

\subsection{Noise modelled on $H^{-1}(\partial \Omega)$}	
	We start with the fact that $(H^{-1}(\partial\Omega))^3$ is a Hilbert space and let $\{\mathbf{e}_n : n\in\N\}$ be an orthonormal basis of $(H^{-1}(\partial\Omega))^3$. Recall that our bilinear form with corrupted data are defined as 
	\begin{align}
	\mathcal N^{A}_{\mu,\gamma}(f,g)=&\int_{\partial\Omega}(\Lambda^{A}_{\mu,\gamma}(f)\times\nu)\cdot \overline{g}~dS+\sum_{\alpha\in\N^2}(f|\mathbf{e}_{\alpha_1})(g|\mathbf{e}_{\alpha_2})X_\alpha
	\label{id:noise_admittance}\\
	\mathcal N^{I}_{\mu,\gamma}(f,g)=&\int_{\partial\Omega}\overline{\Lambda^{I}_{\mu,\gamma}(f)}\cdot (g\times\nu)~dS+\sum_{\alpha\in\N^2}(f|\mathbf{e}_{\alpha_1})(g|\mathbf{e}_{\alpha_2})X_\alpha
	\label{id:noise_impedance}
	\end{align}
	for  $f, g\in H^{-1/2}(\Div;\partial\Omega)\subset (H^{-1}(\partial\Omega))^3$, where $\alpha=(\alpha_1,\alpha_2)\in \mathbb{N}^2$ and $(\phi|\psi)$ denotes the inner product in $(H^{-1}(\partial\Omega))^3$.
	
	Then we have the following lemma after replacing $L^2(\partial\Omega)$ by $(	H^{-1}(\partial\Omega))^3$ in \cite[Lemma 2.3]{caro2017calderon}.
	\begin{lemma}\label{some lemma} There exists a complete probability space $(\Pi,\mathcal H,\mathbb P)$, and a countable family $\{X_\alpha : \alpha\in \N^2\}$ of independent complex random variables satisfying \eqref{eqn:GRV}. Moreover, for every $f, g\in (H^{-1}(\partial\Omega))^3$ we have that 
		\[\E\left|\sum_{\alpha\in\N^2}(f|\mathbf{e}_{\alpha_1})(g|\mathbf{e}_{\alpha_2})X_\alpha\right|^2=\|f\|^2_{(H^{-1}(\partial\Omega))^3}\|g\|^2_{(H^{-1}(\partial\Omega))^3}.\]
	\end{lemma}
	
	Since the $(H^{-1}(\partial\Omega))^3$-norm is bounded by the $H^{-1/2}(\Div,\partial\Omega)$-norm, immediately, we obtain the boundedness of the operators $\mathcal N^{A}_{\mu,\gamma}$ and $\mathcal N^{I}_{\mu,\gamma}$ from the space $H^{-1/2}(\Div;\partial\Omega)\times H^{-1/2}(\Div;\partial\Omega)$ to $L^2(\Pi, \mathcal H, \mathbb P)$. 
	It gives that 
	$\left|\mathcal N^{A}_{\mu,\gamma}(f,g)\right|$, $\left|\mathcal N^{I}_{\mu,\gamma}(f,g)\right|$ are finite almost surely. Moreover, we have the following decay for the covariance.
	
	\begin{lemma}
		The following estimate holds 
		\begin{equation}\label{eqn:f_N_decay}
		\E\left|\sum_{\alpha\in\N^2}(f_N|\mathbf{e}_{\alpha_1})(f_N|\mathbf{e}_{\alpha_2})X_\alpha\right|^ 2=\|f_N\|^4_{(H^{-1}(\partial\Omega))^3}\leq C_{\partial\Omega}N^{-2}.\end{equation}
	\end{lemma}

\begin{proof}
	
    The first equality directly comes from Lemma \ref{some lemma} and the second inequality is obtained as follows. 
	From \eqref{eqn:bdry-fN}, one has the equivalent formula
	\[f_N(z)=c_0^{-1/2}\nu(z)\times W_N(z), \qquad z\in\partial\Omega, \]
	where 
	\[W_N(z):=(F^{-1})^*(\nabla v_N)=M(y)^t\nabla_y v_N(y)|_{y=F^{-1}(z)}.\] 
	Here, $\nu(z)$ is the unit outer normal to $\partial\Omega$ while $\nu(y)$ in \eqref{eqn:bdry-fN} is the unit outer normal to $\partial\tilde\Omega$.

	It is easy to verify 
	\[\nabla_z\times W_N (z)=0 .\]
	For $\varphi\in (H^1(\partial\Omega))^3$, 
	\[\int_{\Omega}\nabla\times W_N\cdot \varphi_e-W_N\cdot\nabla\times\varphi_e~dz=\int_{\partial\Omega}f_N\cdot\varphi~dS, \]
	where $\varphi_e\in (H^{3/2}(\Omega))^3$ is the extension such that $\varphi=\nu\times\varphi_e|_{\partial\Omega}\times\nu$. The first term of the left hand side vanishes by above. For the second term of the left hand side, after a change of variable and passing the derivative, we have
	\[\begin{split}\int_{\partial\Omega}f_N\cdot\varphi~dS=&-\int_\Omega W_N\cdot\nabla\times\varphi~dz\\
	=&-\int_{\Omega}\left(\frac{\partial y}{\partial z}\right)^t\left(\nabla v_N\circ F^{-1}\right)(z)\cdot (\nabla_z\times\varphi(z))~dz\\
	=&-\int_{\widetilde\Omega} \nabla v_N(y)\cdot  (\nabla_y \times \tilde\varphi(y)	)
	\  \det\left(\frac{\partial z}{\partial y}\right) ~dy\\
	=&-\int_{\partial\widetilde\Omega}  v_N\nu\cdot (\nabla_y\times\widetilde\varphi(y))~dS , 
	\end{split}\]
	where $\widetilde\varphi$ is the push-forward of $\varphi$ by $F$ given by
	\[\widetilde\varphi=(M^t)^{-1}(y)\varphi(F(y)).\]
	Finally, it is not hard to see that
	\[\int_{\partial\Omega}f_N\cdot\varphi~dS\leq C\|v_N\|_{(L^2(\partial\widetilde\Omega))^3}\|\varphi\|_{(H^1(\partial\Omega))^3}.\]
	Therefore,
	\[\|f_N\|_{(H^{-1}(\partial\Omega))^3}\leq C\|v_N\|_{(L^2(\partial\widetilde\Omega))^3}\leq C\|v_N\|_{(H^{1/2}(\widetilde\Omega))^3}\leq C_{\p\Omega}N^{-1/2} \]
	which gives \eqref{eqn:f_N_decay}.  
	
\end{proof}

We state one crucial result from \cite{caro2017calderon} which also works for the vector-valued functions in this paper. This result will lead to the unique determination and the rate of convergence of parameters for both Maxwell and elasticity systems with corrupted data.  	
\begin{proposition}[Lemma 2.5 in \cite{caro2017calderon}]\label{Prop auxilliary in max}
 Let $(X,\Sigma,m)$ be a measure space and $\{f_n\}_{n=1}^\infty$ be a vector-valued sequence in $(L^s(X,\Sigma,m))^3$ for $s\in [1,\infty)$. Assume that $f_n \to f $ in  $(L^s(X,\Sigma,m))^3$ for some $f\in  (L^s(X,\Sigma,m))^3$ and there exists a sequence of positive numbers $\{\lambda_n\}_{n=1}^\infty\subset \R_+$ with $\lambda_n \to 0$ as $n\to \infty$ such that 
		
		$$
		\sum_{n=1}^\infty \dfrac{1}{\lambda_n^s}\int_X |f_n-f|^s dm < \infty.
		$$
		Then one has $f_n \to f$ for almost every $x \in X$. 
		
		Suppose furthermore that $m(X) <\infty$. Then, for every $\epsilon>0$, there exists a $n_0\in \N$ such that 
		$$
		m \{x\in X: \ |f_n(x)-f(x)|\leq \lambda_n\} \geq m(X)-\epsilon, \text{ for }n\geq n_0.
		$$
 
\end{proposition}
\begin{remark}\label{rmk:n0}
The $n_0$ in the second part of the statement should satisfy 
\[\sum_{n=n_0}^\infty\frac{1}{\lambda_n^s}\int_X|f_n-f|^s~dm\leq \epsilon.\]
\end{remark}
\bigskip


	
	\begin{proof}[Proof of Theorem \ref{thm:corrupted}]
	The part (1) is a consequence of the first part of Proposition \ref{Prop auxilliary in max} to the sequence $\left\{\sum(f_{N}| \mathbf{e}_{\alpha_1}) (\overline{f_{N}}| \mathbf{e}_{\alpha_2}) X_\alpha~:~N\in\mathbb N\backslash\{0\}\right\}$ with $\lambda_N=N^{-\theta}$.\\
	
	To prove part (2) of Theorem \ref{thm:corrupted}, again we take $\lambda_N=N^{-\theta/2}$. Applying the second part of Proposition \ref{Prop auxilliary in max} to the sequence $\{\sum(f_{N}| \mathbf{e}_{\alpha_1}) (\overline{f_{N}}| \mathbf{e}_{\alpha_2}) X_\alpha:N\in\mathbb N\backslash\{0\}\}$, and using \eqref{eqn:f_N_decay}, we obtain
	\[\mathbb P\left\{\left|\sum(f_{N}| \mathbf{e}_{\alpha_1}) (\overline{f_{N}}| \mathbf{e}_{\alpha_2}) X_\alpha\right|\leq N^{-\theta/2}\right\}\geq 1-\epsilon\]
	for $N\geq N_0$, where $N_0$ is as in Remark \ref{rmk:n0}, that is, we need
	\[\sum_{N=N_0}^\infty \frac{C_{\partial\Omega}^2}{N^{2-\theta}}\leq\epsilon.\]
	This holds whenever
	\[(N_0-1)^{1-\theta}>\frac{C_{\partial\Omega}^2}{\epsilon(1-\theta)},\] which gives $N_0\geq c\epsilon^{-\frac1{1-\theta}}$. 
	Lastly, we see that there exist $C_\gamma>0$ and $C_\mu>0$ such that 
	\[\left\{\left|\sum(f_{N}| \mathbf{e}_{\alpha_1}) (\overline{f_{N}}| \mathbf{e}_{\alpha_2}) X_\alpha\right|\leq N^{-\theta/2}\right\}\subset\left\{|\mathcal N^A_{\mu,\gamma}(f_N, {f_N})-\gamma(P)|\leq C_\gamma N^{-\theta/ 2}\right\}\]
	and
	\[\left\{\left|\sum(f_{N}| \mathbf{e}_{\alpha_1}) (\overline{f_{N}}| \mathbf{e}_{\alpha_2}) X_\alpha\right|\leq N^{-\theta/2}\right\}\subset\left\{|\mathcal N^I_{\mu,\gamma}(f_N, {f_N})-\mu(P)|\leq C_\mu N^{-\theta/ 2}\right\},\]
	respectively. This completes the proof.
	\end{proof}

\subsection{Noise modelled on $L^2(\partial \Omega)$}
The noisy admittance and impedance data, $\mathcal{N}^A_{\mu, \gamma}$ and $\mathcal{N}^I_{\mu, \gamma}$ respectively, are defined in the level of $L^2(\partial \Omega)$ exactly in the same way as in \eqref{id:noise_admittance} and \eqref{id:noise_impedance} with the exception of some details. The sequence $\{ \mathbf{e}_n : n\in \N  \}$ is an orthonormal basis of $(L^2(\partial \Omega))^3$, the inner product $(\phi|\psi) = \int_{\partial \Omega} \phi \cdot \overline{\psi}dS$, and finally $f, g \in H^{1/2}(\Div, \partial \Omega)$. To make rigorous sense of this definition, we will assume the boundary of the domain to be locally defined by the graph of $C^{1,1}$ functions.

\begin{lemma}\label{lem:covariance} There exists a complete probability space $(\Pi,\mathcal H,\mathbb P)$, and a countable family $\{X_\alpha : \alpha\in \N^2\}$ of independent complex random variables satisfying \eqref{eqn:GRV}. Moreover, for every $f, g\in (L^2(\partial\Omega))^3$ we have that 
		\[\E\left|\sum_{\alpha\in\N^2}(f|\mathbf{e}_{\alpha_1})(g|\mathbf{e}_{\alpha_2})X_\alpha\right|^2=\|f\|^2_{(L^2(\partial\Omega))^3}\|g\|^2_{(L^2(\partial\Omega))^3}.\]
	\end{lemma}
	
	\begin{lemma}\label{lem:f_N_norm}
		The following estimate holds 
		\begin{equation}\label{eqn:f_N_behaviour}
		\E\left|\sum_{\alpha\in\N^2}(f_N|\mathbf{e}_{\alpha_1})(f_N|\mathbf{e}_{\alpha_2})X_\alpha\right|^ 2=\|f_N\|^4_{(L^2(\partial\Omega))^3}\leq C_{\partial\Omega}.
		\end{equation}
	\end{lemma}
	
\begin{proof}
To compute the $L^2$-norm of $f_N$, we could just take the part of $\partial \Omega$ inside the ball of radius $\rho$ and center $P$ since $f_N$ vanishes outside. This part of $\partial \Omega$ could be flatten and there the following identity would hold if $N^{-1/2} < 2r $
\[f_N = \nu \times E|_{\partial \Omega} = DF (\vec{e}_d \times \nabla v_N)|_{\partial \wt \Omega}. \]
A straightforward computation shows that
\[\vec{e}_d \times \nabla v_N(x', 0) = e^{iN \alpha \cdot (x', 0)} \big[ N^{1/2} \vec{e}_d \times \nabla \psi \big( N^{1/2} (x',0) \big) + i N (\vec{e}_d \times \alpha) \psi \big( N^{1/2} (x',0) \big) \big],\]
where $\psi(x) = \eta(|x'|) \eta(x_d)$. On the other hand, note that
\[ DF(x',0) = \left[ \begin{array}{ c c}
I_{d - 1} & 0 \\
\nabla^\prime \phi (x' + p') & 1
\end{array} \right], \qquad \vec{e}_d \times \alpha = \frac{(1 + |\nabla' \phi(p')|^2)}{|\nabla' \phi (p')|} \left[ \begin{array}{c}
-\partial_2 \phi (p') \\ \partial_1 \phi (p') \\ 0
\end{array} \right], \]
which implies that $DF(0,0) (\vec{e}_d \times \alpha) = 0$. Therefore, for $|x'| < 2r$, we have that
\begin{equation}
\begin{aligned}
f_N (F(x',0)) = &\ e^{iN \alpha \cdot (x', 0)} \big[ N^{1/2} DF(x',0) ( \vec{e}_d \times \nabla \psi) \big( N^{1/2} (x',0) \big) \\
 & + i N \big( DF(x',0) - DF(0,0) \big) (\vec{e}_d \times \alpha) \psi \big( N^{1/2} (x',0) \big) \big].
\end{aligned}
\label{term:f_N}
\end{equation}
Thus,
\begin{align*}
\| f_N \|_{(L^2(\partial \Omega))^3} \lesssim &\ N^{1/2} \| DF(x',0) ( \vec{e}_d \times \nabla \psi) \big( N^{1/2} (x',0) \big) \|_{(L^2(\R^2))^3} \\
& + N \| \big( DF(x',0) - DF(0,0) \big) (\vec{e}_d \times \alpha) \psi \big( N^{1/2} (x',0) \big) \|_{(L^2(\R^2))^3}.
\end{align*}
The first term on the right hand side is bounded by a constant independent of $N$ because the rate of shrinking of the support of $( \vec{e}_d \times \nabla \psi) (N^{1/2} x',0) $. To ensure that the second term is also bounded by a constant independent of $N$ we need an extra cancellation beside the shrinking of the support. This cancellation comes from the inequality $|DF(x',0) - DF(0,0)| \lesssim |x'|$, which is a consequence of the fact that $\partial \Omega$ is locally described by $C^{1,1}$ functions.
\end{proof}

\begin{lemma}\label{lem:avg}
We have that, for $T\geq 1$, there exists a $C > 0$ so that
\[\mathbb{E} \bigg| \frac{1}{T} \int_T^{2T} \sum_{\alpha \in \N^2} (f_{t^2}| \mathbf{e}_{\alpha_1}) (\overline{f_{t^2}}| \mathbf{e}_{\alpha_2}) X_\alpha \, dt \bigg|^2 \leq \frac{C}{T^{2/3}}.\]
The constant $C$ depends on upper bounds for the $C^{1,1}$ norm of the functions describing locally the boundary of $\partial \Omega$.
\end{lemma}

\begin{proof}
One can check that
\[\mathbb{E} \bigg| \frac{1}{T} \int_T^{2T} \sum_{\alpha \in \N^2} (f_{t^2}| \mathbf{e}_{\alpha_1}) (\overline{f_{t^2}}| \mathbf{e}_{\alpha_2}) X_\alpha \, dt \bigg|^2 = \frac{1}{T^2} \int_{Q_T} \big|(f_{s^2}|f_{t^2})\big|^2 \,d(s,t),\]
where $Q_T = [T, 2T]\times[ T, 2T]$. Consider $S \in (0, T/2)$ to be chosen later and split $Q_T$ in the sets
\begin{align*}
D(S) &= \{ (s,t) \in Q_T : t-S \leq s \leq t+S \}, \\
L(S) &= \{ (s,t) \in Q_T : T \leq s < t - S \}, \\
R(S) &= \{ (s,t) \in Q_T : t + S < s \leq 2T \}.
\end{align*}
Using Cauchy--Schwarz and the Lemma \ref{lem:f_N_norm}, we have that
\begin{equation}
\frac{1}{T^2} \int_{D(S)} \big|(f_{s^2}|f_{t^2})\big|^2 \,d(s,t) \lesssim \frac{|D(S)|}{T^2} \simeq \frac{S}{T},
\label{es:diagonal}
\end{equation}
since $|D(S)|$, the Lebesgue measure of $D(S)$ is of the order $ST$. We are now going to study the other pieces $L(S)$ and $R(S)$. Start by noticing that, using the expression \eqref{term:f_N}, the inner product $(f_{s^2} | f_{t^2})$ can be written as a sum of terms of the form
\begin{equation}
s t \int_{\R^2} e^{i(s^2 - t^2) \alpha' \cdot x'} a(x';s) b(x'; t) \, dx',
\label{int:oscillatory}
\end{equation}
where $| \partial^\beta a(x';s) | \lesssim (1 + s)^{|\beta|} \chi(s x')$ and $| \partial^\beta b(x';t) | \lesssim (1 + t)^{|\beta|} \chi(t x')$ with $\chi$ a compactly supported function in $\R^2$ and $\beta \in \N^2$ for $|\beta| \leq 1$. Since $D(S)$ contains the stationary points of the oscillatory integral \eqref{int:oscillatory}, we have that, in $L(S)$ and $R(S)$, its phase is non-stationary. Then, write
\[e^{i(s^2 - t^2) \alpha' \cdot x'} = \frac{-i}{|\alpha'|^2 (s^2 - t^2)} \alpha' \cdot \nabla  e^{i(s^2 - t^2) \alpha' \cdot x'} \]
in order to count the oscillations. Thus, the absolute value of \eqref{int:oscillatory} can be bounded, modulo a multiplicative constant, by
\[\frac{st}{|s^2 - t^2|} \int_{\R^2} |\nabla a(x'; s)| |b(x'; t)| + |a(x'; s)| |\nabla b(x'; t)| \, dx' \]
which in term is bounded, again modulo a multiplicative constant, by
\begin{equation}
\frac{1 + s + t}{|s^2 - t^2|} st \int_{\R^2} \chi(sx') \chi(tx')\, dx' \lesssim \frac{1 + s + t}{|s^2 - t^2|}.
\label{es:boundednessLR}
\end{equation}
In the last inequality, we have used Cauchy--Schwarz. In $R(S)$, $s^2 - t^2 > 0$ since
\[ s^2 - t^2 > (t +S)^2- t^2 = 2St + S^2 > tS \geq ST. \]
Hence, $|s^2 - t^2| \geq ST$.
On the other hand, in $L(S)$, $t^2 - s^2 > 0$ since
\[ t^2 - s^2 > t^2- (t - S)^2 = 2St - S^2 > tS \geq ST. \]
Again, $|s^2 - t^2| \geq ST$.
Thus, by the fact that $(f_{s^2} | f_{t^2})$ can be written as a sum of terms of the form \eqref{int:oscillatory}, and these in turn can be bounded by the right-hand side of \eqref{es:boundednessLR}, we have that
\begin{equation}
\frac{1}{T^2} \int_{L(S) \cup R(S)} \big|(f_{s^2}|f_{t^2})\big|^2 ~d(s,t) \lesssim \frac{|L(S) \cup R(S)|}{T^2} \frac{T^2}{S^2 T^2} \lesssim \frac{1}{S^2}
\label{in:off_diagonal}
\end{equation}
since $|L(S) \cup R(S)| \lesssim T^2$. Choosing $S = T^{1/3}$ to make the decays in \eqref{es:diagonal} and \eqref{in:off_diagonal} of the same order, we have the inequality stated in the lemma.
\end{proof}

\begin{proof}[Proof of the Theorem \ref{thm:corrupted-L2}]
The proof basically follows the proof of Theorem \ref{thm:corrupted} by applying Proposition \ref{Prop auxilliary in max} to the sequence of random variables
\[\left\{\frac{1}{T_N}\int_{T_N}^{2T_N}\sum_{\alpha \in \N^2} (f_{t^2}| \mathbf{e}_{\alpha_1}) (\overline{f_{t^2}}| \mathbf{e}_{\alpha_2}) X_\alpha ~ dt:N\in\mathbb N\backslash\{0\}\right\}\]
and by applying
\[\mathbb E\left|\frac{1}{T_N}\int_{T_N}^{2T_N}\sum_{\alpha \in \N^2} (f_{t^2}| \mathbf{e}_{\alpha_1}) (\overline{f_{t^2}}| \mathbf{e}_{\alpha_2}) X_\alpha ~ dt \right|^2\leq \frac{C}{N^{2+\theta}}\rightarrow 0\]
as $N\rightarrow\infty$, 
obtained using Lemma \ref{lem:avg} and $\lambda_N=N^{-\theta/2}$. Note that the $C_{\partial\Omega}$ (used in control $N_0$) is replaced by constants in Lemma \ref{lem:avg}, which depend on $\partial\Omega$, lower bounds for $\varepsilon_0$ and $\mu_0$, and upper bounds for $\|\gamma\|_{\rm Lip(\overline{\Omega})}$ and $\|\mu\|_{\rm Lip(\overline\Omega)}$, respectively.
\end{proof}

\section{Boundary determination of Lam\'e moduli with corrupted data}\label{sec:4}
In this section, assuming that the data has measurement error as in section \ref{sec:3}, we reconstruct the boundary value of Lam\'e parameters and its rates of convergence formula for the isotropic elasticity system. 

Hereafter, we will consider the problem in $\R^3$. 
Let $\Omega \subset \R^3$ be a bounded domain, $\lambda(x)$ and $\mu(x)$ be the Lam\'e parameters satisfying the strong convexity condition in \eqref{strong convexity}. The regularity assumptions of the boundary $\partial \Omega$ and the Lam\'e parameters $(\lambda,\mu)$ will be described later. 

We use the same notations as in Section \ref{sec:2}. Given $P=(p',p_3) \in \p \Omega$ and $x=(x',x_3)$, let $\phi:\R^2 \to \R$ be the Lipschitz function and $(z',z_3)=F(x',x_3)=(x'+p',x_3+\phi(x'+p'))$ be the boundary flatten map near $P\in \p \Omega$. The matrix $M$ is defined in \eqref{eqn:Jacobi} with $\det M(x)=1$.
Let $\wt \Omega=F^{-1}(\Omega)$.

Let $u$ be the solution to the elasticity system \eqref{Lame system} associated to the tensor $\mathbf C$. By a change of coordinates, the function
$\wt u(x):=u(F(x))$ solves a new elasticity system  
\begin{align}
\nabla \cdot (\widetilde{\mathbf{C}}\nabla \wt u)=0 \text{ in }\wt\Omega ,
\end{align}
where we have utilized that 
$$
0=\int_\Omega \mathbf{C}\nabla u:\nabla \phi~ dz = \int_{\wt \Omega }\widetilde{\mathbf{C}}\nabla \wt u : \nabla \wt \phi ~dx, \text{ for any smooth test function }\phi.
$$
Here $\tilde{\mathbf{C}}$ is the elastic tensor expressed as
\begin{align}\label{reprensentation of tilde C}
\widetilde{\mathbf{C}}(x)=M(x)\otimes\mathbf{C}(F(x))\otimes M(x)^t,
\end{align}
where $\otimes$ denotes the multiplication between a fourth-order rank tensor and a matrix. In particular, the function $\widetilde{\mathbf{C}}=(\widetilde{C}_{iqkp})_{1\leq i,q,k,p\leq 3}$ can be explicitly written as 
\begin{align}\label{componentwise representation}
\widetilde{C}_{iqkp}=\left.\sum_{l,j=1}^3 C_{ijkl}\dfrac{\p x_p}{\p z_l}\dfrac{\p x_q}{\p z_j}\right|_{z=F(x)}.
\end{align} 
Moreover, $\widetilde{\mathbf{C}}$ satisfies the strong convexity condition \eqref{strong convexity}, but with a different positive lower bound. Note that the new elastic tensor $\widetilde{\mathbf{C}}$ will lose the minor symmetric property, but we can still reconstruct its coefficients at the boundary.
Use a change of variable again, then we have 
\begin{align}\label{DN identity}
\int_{\p \Omega}\Lambda_{\mathbf{C}}f\cdot \overline{f}~dS=&\int_\Omega \mathbf{C}\nabla u:\nabla \overline{u}~ dz =\int_{\wt \Omega}\widetilde{\mathbf{C}} \nabla \wt u :\nabla \overline{\wt u}~ dx,
\end{align}
where $\Lambda_{\mathbf C}$ is the Dirichlet-to-Neumann map defined by \eqref{eqn:DN-map_elasticity} and $:$ denotes the Frobenius product between two matrices.

\subsection{Approximate solution and elliptic estimate}
We first give a reconstruction formula for the Lam\'e parameters $\lambda$ and $\mu$ on the surface.

Recall that $\eta: \R\to [0,1]$ be a smooth cutoff function given in Section \ref{sec:2}.
Let $\omega\in\R^3$, depending on $x'$, be chosen such that
\begin{equation}\label{eqn:w-en}
\begin{split}&|M(x',0)^t\omega|= |M(x',0)^t\vec e_3|,\\ 
&\omega\cdot M(x',0) M(x',0)^t\vec e_3=0,\end{split}
\end{equation}
where $\vec e_3=(0,0,1)$.

Given any vector $\mathbf{a}=(a_1,a_2,a_3)\in \C^3$, for any integer $N\geq 1$, we define a family of approximation solutions of $\wt u$ by
\begin{align*}
\wt G_{N}(y)=\eta(N^{1/2}|y'|)\eta(N^{1/2}y_3)e^{N(\sqrt{-1} \omega - \vec e_3)\cdot (y-(x',0))} \mathbf{a}
\end{align*}
in a similar spirit for the Maxwell system in section \ref{sec. 2.2}, see also \cite[Section 2.3.2.1]{tanuma2007stroh}.
Because of the need to use $i$ as a summation index, we let $\sqrt{-1}$ denote the imaginary unit.
From now on, without loss of generality, we assume that $x'=0$. Then $\omega$ satisfies \eqref{eqn:w-en} and $\omega=(\omega_1,\omega_2,0)$.
Similar to the notations introduced in section \ref{sec:2}, we denote 
$$
   \psi_N(y) = \eta(N^{1/2}|y'|)\eta(N^{1/2}y_3),\ \ e_N(y) =  e^{N(\sqrt{-1} \omega - \vec e_3)\cdot y} ,
$$
then we can express $\wt G_N$ as
\begin{align}\label{0-order approx. sol}
\wt G_{N}(y)  =\psi_N(y)e_N(y) \mathbf{a}.
\end{align}

In what follows, 
we first apply the gradient of the approximate solution 
$\{\nabla \wt G_N\}_{N=1}^\infty$ in the integral \eqref{elasticity expansion} in Lemma \ref{Lem of approx sol change of vari} and then find out that its first term dominates the whole behavior. 
This observation will play an essential role in providing the reconstruction formula for $\widetilde{\mathbf{C}}(0)$ in section \ref{sec 4.2} assuming the boundary measurements are corrupted. 

\begin{lemma}\label{Lem of approx sol change of vari}
	Let $\lambda,\mu$ be the Lipschitz continuous Lam\'e moduli satisfying the strong convexity condition \eqref{strong convexity}.  The four tensor $\widetilde{\mathbf{C}}$ is defined by \eqref{reprensentation of tilde C} and $\nabla' \phi(p')$ exists. Then we have 
	\begin{align}\label{elasticity expansion}
	&\int_{\wt\Omega}\widetilde{\mathbf{C}}\nabla \wt G_{ N}:\nabla \overline{\wt G_{N}} ~dy\\ 
	 =&\ \notag  \sum_{i,j,k,l=1}^3 C_{ijkl}(F(0)) a_k \overline{a_i}\left(\sum_{p,q=1}^2\dfrac{\p y_p}{\p z_l}(0)\dfrac{\p y_q}{\p z_j}(0)\omega_p \omega_q + \dfrac{\p y_3}{\p z_l}(0)\dfrac{\p y_3}{\p z_j}(0)\right) \\
     \quad &\notag \times \int_{\R^2}\eta (|y'|)^2dy'  + O\left(e^{ -\frac{1}{2}N^{1/2}}\right) \\
     \quad& \notag+O\left(N^{-1/2} +\left( N \int_{|y'|\leq N^{-1/2}}|\nabla' \phi(y'+p')-\nabla'\phi(p')|^2~dy'\right)^{1/2} \right),
	\end{align}
	where $\wt G_{N}$ is the approximation solution defined by \eqref{0-order approx. sol} and recall that $\nabla'\phi:=(\partial_1\phi, \partial_2\phi)^t$.
\end{lemma}

For the flat case (i.e., $0\in \p \Omega$ with $\Omega=\{z_3>0\}$ near $0$), the previous lemma was proved in \cite[Section 2]{tanuma2007stroh}. The first term in the right hand side of \eqref{elasticity expansion} is the dominant term of the boundary determination, while the remaining parts are lower order terms. For the completeness of the paper, we provide a detailed proof below.

\begin{proof}[Proof of Lemma \ref{Lem of approx sol change of vari}]
	Following the idea in the proof of \cite[Lemma 1]{brown2001recovering}, we first note that 
	\begin{align}\label{integral formula}
	\int_{\wt \Omega} \widetilde{\mathbf{C}}\nabla \wt G_{N}:\nabla \overline{\wt G_{N}} ~dy =  \sum_{i,q,k,p=1}^3   \int_{\wt \Omega}\widetilde{C}_{iqkp}(y)\dfrac{\p (\wt G_{N})_k}{\p y_p}\dfrac{\p (\overline{\wt G_{N}})_i}{\p y_q}~dy,
	\end{align}
	where $\wt G_N = ((\wt G_N)_1, (\wt G_N)_2,(\wt G_N)_3)$.
	For $k=1,2,3$, by a direct computation, the partial derivatives of $(\wt G_N)_k$ are 
	\begin{align}\label{1-2 components}
	\dfrac{\p (\wt G_{N})_k}{\p y_p}=  \left( N^{1/2}\eta'(N^{1/2} |y'|)\eta(N^{1/2}y_3){y_p\over |y'|} + \sqrt{-1} N \omega_p \psi_N(y)\right)e_N(y)a_k,
	\end{align}
	for $p=1,2$ and 
	\begin{align}\label{3rd component}
	\dfrac{\p (\wt G_{N})_k}{\p y_3}=  \left( N^{1/2}\eta (N^{1/2} |y'|)\eta'(N^{1/2}y_3) -N  \psi_N(y)\right)e_N(y)a_k.
	\end{align}
	
	Next, substituting \eqref{componentwise representation}, \eqref{1-2 components} and \eqref{3rd component} into the identity \eqref{integral formula}, then one  obtain 
	\begin{align*}
	 &\int_{\wt \Omega} \widetilde{\mathbf{C}}\nabla \wt G_{N}:\nabla \overline{\wt G_{N}} ~dy \\
	 =&\ \sum_{i,j,k,l,q,p=1}^3 \int_{\wt \Omega} C_{ijkl}(F(y))\dfrac{\p y_p}{\p z_l}\dfrac{\p y_q}{\p z_j}\dfrac{\p (\wt G_{N})_k}{\p y_p}\dfrac{\p (\overline{\wt G_{N}})_i}{\p y_q}~dy\\
	 =:&\  I + II + III + IV,
	\end{align*}	
	where
	\begin{align*}
	I &=  N^2 \sum_{i,j,k,l=1}^3\int_{\wt \Omega}C_{ijkl}(F(0)) \left(\sum_{p,q=1}^2\dfrac{\p y_p}{\p z_l}(0)\dfrac{\p y_q}{\p z_j}(0)\omega_p \omega_q + \dfrac{\p y_3}{\p z_l}(0)\dfrac{\p y_3}{\p z_j}(0) \right)\\
	& \quad \times \eta(N^{1/2}|y'|)^2\eta(N^{1/2}y_3)^2 e^{-2Ny_3} a_k \overline{a_i} ~dy,\\
	II &=   N^2 \sum_{i, k=1}^3\left[ \sum_{p,q=1}^2\int_{\wt \Omega}\left(\wt C_{iqkp}(y)-\wt C_{iqkp}(0)\right) \omega_p \omega_q   +\int_{\wt \Omega}\left(\wt C_{i3k3}(y)-\wt C_{i3k3}(0)\right)\right]\\
	& \quad \times \eta(N^{1/2}|y'|)^2\eta(N^{1/2}y_3)^2 a_k \overline{a_i}~ dy, \\
	III &=  N^{3/2}\sum_{i,j,k,l=1}^3\int_{\wt \Omega}C_{ijkl}(F(y)) 
	\Big(-2\sum_{p =1}^2\dfrac{\p y_p}{\p z_l}\dfrac{\p y_3}{\p z_j}\psi_N(y) \eta'(N^{1/2}|y'|) \\
	&\quad\times \eta (N^{1/2}y_3) {y_p\over|y'|}   -2 \dfrac{\p y_3}{\p z_l}\dfrac{\p y_3}{\p z_j}\psi_N(y)\eta(N^{1/2}|y'|)\eta'(My_3)\Big) e^{-2Ny_3} a_k \overline{a_i} ~dy,
	\end{align*}

	and 
	\begin{align*}
	IV &=  N\sum_{i,j,k,l=1}^3\int_{\wt \Omega} C_{ijkl}(F(y))\Big(\sum_{p,q=1}^2\dfrac{\p y_p}{\p z_l}\dfrac{\p y_q}{\p z_j} \eta'(N^{1/2}|y'|)^2 \eta(N^{1/2} y_3)^2 {y_py_q\over |y'|^2}  \\
	& \quad \left.+ \sum_{p =1}^2\dfrac{\p y_p}{\p z_l}\dfrac{\p y_3}{\p z_j} 2\eta'(N^{1/2}|y'|)\eta(N^{1/2}|y'|)\eta(N^{1/2}y_3)\eta'(N^{1/2}y_3) {y_p\over |y'|}\right.\\
	&\quad + \dfrac{\p y_3}{\p z_l}\dfrac{\p y_3}{\p z_j}  \eta'(N^{1/2}y_3)^2\eta(N^{1/2}|y'|)^2\Big)
	e^{-2Ny_3} a_k \overline{a_i} ~dy.
	\end{align*}
	
	We will show that $I$ is the dominant term and $II$, $III$, $IV$ are remainder terms in the following arguments. We first estimate $I$. By using the integration by parts with respect to the $y_3$ variable and applying change of variables, we obtain  	
	\begin{align*}
	I  
	& = \sum_{i,j,k,l=1}^3 C_{ijkl}(F(0)) a_k \overline{a_i}\left(\sum_{p,q=1}^2\dfrac{\p y_p}{\p z_l}(0)\dfrac{\p y_q}{\p z_j}(0)\omega_p \omega_q + \dfrac{\p y_3}{\p z_l}(0)\dfrac{\p y_3}{\p z_j}(0)\right) \\
	&\quad \times \int_{\R^2}\eta(y')^2~dy'  + O\left( e^{ -\frac{1}{2}N^{1/2}}\right).	
	\end{align*}
	
 Secondly, by using change of variables again and following a similar argument as in the proof of \cite[Lemma 1]{brown2001recovering}, one can derive that 
	$$
	III = O\left(N^{-1/2}\right) \text{ and }IV = O\left(N^{-1/2}\right).
	$$
	
	Finally, for the second term $II$, the triangle inequality yields that 
	\begin{align}\label{estimate II}
	 |II|   
		\leq  II_1 +II_2,
	\end{align}
where 
\begin{align*}
II_1   \lesssim &\ N^2\|\nabla F^{-1}\|^{2}_\infty \sum_{i,j,k,l=1}^3\int_{\wt \Omega}\left| C_{ijkl}(F(y))-  C_{ijkl}(F(y',0))\right| \\ 
 & \times\eta(N^{1/2}|y'|)^2\eta(N^{1/2}y_3)^2e^{-2Ny_3}~ dy, \\
II_2  \lesssim &\ N \sum_{i, k =1}^3\int _{\R^2}\Big(\sum^2_{p,q=1}|\wt C_{iqkp}(y',0)-\wt C_{iqkp}(0)  |  \\
&  + |\wt C_{i3k3}(y',0)-\wt C_{i3k3}(0)|\Big) \eta(N^{1/2}|y'|)^2~dy',
\end{align*}
for some constant $C>0$ independent of $N$. Here we have utilized that $|\omega_p|\leq 1$ for $p=1,2$ (recalling that $\omega=(\omega_1,\omega_2,0)$ is a unit vector) and $a_j$'s are complex numbers for $j=1,2,3$. 

To establish \eqref{estimate II}, we will estimate $II_1$ and $II_2$ separately.
For $II_1$, 
we choose a constant $\lambda>0$ and split the region of integral into two parts, namely, $\{y_3>\lambda\}$ and $\{y_3<\lambda\}$. 
Thus, one obtains, by following a similar argument as in \eqref{equ:max}, that
\begin{align}\label{estimate II_1}
|II_1| \lesssim o(1) 
\end{align}	
when $N\rightarrow \infty$.

On the other hand, for $II_2$, by Cauchy-Schwartz inequality, one can derive 
\begin{align*}
|II_2| &\lesssim \Big( N \int_{|y'|\leq N^{-1/2}}  |\wt C_{iqkp}(y',0)-\wt C_{iqkp}(0) |^2 + |\wt C_{i3k3}(y',0)-\wt C_{i3k3}(0) |^2 ~dy'\Big)^{1/2} \notag \\
&\leq  \Big(N \int_{|y'|\leq N^{-1/2}}|\nabla'\phi(y'+p')-\nabla'\phi(p')|^2  ~dy'\Big)^{1/2}.
\end{align*}	 
Using \eqref{lim:lebesgue_point}, it leads to
\begin{align}\label{estimate II_2}
|II_2| \lesssim	o(1). 
\end{align}
We substitute \eqref{estimate II_1} and \eqref{estimate II_2} into \eqref{estimate II}. We combine the estimates for $I$ to $IV$, then we complete the proof.
	
\end{proof}

We denote 
$$
    \kappa:= \int_{\R^2}\eta(|y'|)^2 ~dy',
$$
by a direct computation and 
let $N\rightarrow \infty$, then the main term $I$ satisfies
\begin{align}\label{reconstruction formula of elasticity}
\notag I &\rightarrow  \kappa \sum_{i,j,k,l=1}^3 C_{ijkl}(F(0)) a_k \overline{a_i}\left(\sum_{p,q=1}^2\dfrac{\p y_p}{\p z_l}(0)\dfrac{\p y_q}{\p z_j}(0)\omega_p \omega_q + \dfrac{\p y_3}{\p z_l}(0)\dfrac{\p y_3}{\p z_j}(0)\right) \\
&=  \kappa \sum_{i,j=1}^3 Z_{ij}(P)a_i \overline{a_j},
\end{align}
where $Z(P)=(Z_{ij})_{1\leq i,j\leq 3}(P)$ is the $2$-tensor defined by \eqref{eq:Z_ij}.
For more detailed analysis about the boundary reconstruction for the isotropic elasticity system without noisy, we refer readers to \cite[Section 2]{tanuma2007stroh}.

Similar to \cite[Lemma 2.2]{caro2017calderon}, we have an analogues result for the elasticity system.

\begin{lemma}\label{Lem:elliptic estimate for elasticity}
Let $\mathbf{C}$ be a Lipschitz continuous isotropic elastic tensor given by \eqref{elastic four tensor}, which satisfies \eqref{strong convexity}. Let $\widetilde{\mathbf{C}}$ be the elastic four tensor defined by \eqref{reprensentation of tilde C} and $\nabla' \phi(p')$ exists. Let $\wt r_{N}$ be the solution of 
\begin{align*}
\begin{cases}
\nabla \cdot (\widetilde{\mathbf{C}}\nabla \widetilde{r}_{N})= -\nabla \cdot (\widetilde{\mathbf{C}}\nabla \widetilde{G}_{N}) & \text{ in }\widetilde\Omega, \\
\wt r_{N} =0 & \text{ on }\partial \widetilde\Omega,
\end{cases}
\end{align*}
where $\widetilde{G}_{N} \in (H^1(\widetilde{\Omega}))^3$ is the approximate solution defined by \eqref{0-order approx. sol}. If $\nabla' \phi(p')$ exists, then one has 
\begin{align}\label{elliptic estimate for r_N}
&\quad \|\nabla \wt r_{N}\|_{(L^2(\widetilde{\Omega}))^3	}  \notag \\ 
&\lesssim N^{-1/2}  +N^{1/2}\left(\int_{|y'|\leq N^{-1/2}}|\nabla'\phi(y'+p')-\nabla'\phi(p')|^2~ dy'\right)^{1/2}  ,
\end{align}
for some constant $C>0$ independent of $\widetilde{G}_{N}$ and $\widetilde{r}_{N}$.
\end{lemma}

\begin{proof}
The estimate \eqref{elliptic estimate for r_N} holds by using the standard elliptic regularity estimate of $\wt r_{N}$, Hardy's inequality for $\widetilde{u}_N$ and Lemma \ref{Lem of approx sol change of vari}. The detailed proof is the same as the one of \cite[Lemma 2]{brown2001recovering}, thus we refer the interested readers to \cite{brown2001recovering}.

\end{proof}

\subsection{Proof of Theorem \ref{Thm:elastic boundary determination}}\label{sec 4.2}

Let us consider the function $u_N$ with $\wt u_{N}=F^* u_{N}$ and define $\wt u_{N}:=\kappa^{-1/2}(\wt G_{N}+\wt r_{N})$, then $\wt u_{N}\in (H^1(\wt\Omega))^3$ is the solution of 
\begin{align}\label{elastic appro sol}
\nabla \cdot (\widetilde{\mathbf{C}}\nabla \wt u_{N})=0 \text{ in }\wt \Omega \quad\text{ with }\quad \wt u_{N}=\wt G	_{N} \text{ on }\p \wt \Omega.
\end{align}
Denote $f_N=u_N|_{\p \Omega}$.
From formula \eqref{DN identity}, one has
\begin{align}\label{int by parts formula 1}
\kappa^{-1}\int_{\p \wt \Omega} \Lambda_{\wt{\mathbf{C}}}\wt G_{N}\cdot \overline{\wt G_{N}}~dS = \int_{\widetilde{\Omega}}\widetilde{\mathbf{C}} \nabla \wt u_{N}:\nabla \overline{\wt u_{N}}~ dx.
\end{align}

Recall that $(\Pi,\mathcal H, \mathbb P)$ is a complete probability space, and $\{X_\alpha: \alpha\in \mathbb N^2\}$ is a countable family of independent complex Gaussian random variables $X_\alpha:\varpi\in\Pi~\mapsto~X_\alpha(\varpi)\in\C$ as in Section \ref{sec:3} such that \eqref{eqn:GRV} holds with standard expectation of a random variable defined by $\E X=\int_\Pi X d\PP$.
Let $\{\bold e_n:n\in \N\}$ be an orthonormal basis of  $(L^2(\p\Omega))^3$, then we define the noisy data for the isotropic elasticity system via the bilinear form 
\begin{align}\label{noisy bilinear form for elasticity}
\mathcal{N}_{\mathbf{C}}(f,g):=\int_{\p \Omega}\Lambda_{\mathbf{C}}f\cdot \overline{g}~ dS+\sum _{\alpha \in \N^2} (f|\bold e_{\alpha_1})(g|\bold e_{\alpha_2})X_\alpha,
\end{align}
for $f,g\in (H^{1/2}(\p \Omega))^3$, where $\alpha=(\alpha_1,\alpha_2)$ and $(W | w)=\int_{\p \Omega}W \cdot \overline{w} dS\in \C$, for any $W,w\in (L^2(\Omega))^3$.

Next, 
by change of variables, \eqref{int by parts formula 1}, Lemma \ref{Lem of approx sol change of vari}, Lemma \ref{Lem:elliptic estimate for elasticity} and the Cauchy-Schwarz inequality, the equation \eqref{noisy bilinear form for elasticity} yields that  
\begin{align}\label{noisy bilinear 2}
\notag \mathcal{N}_{\mathbf{C}}(f_{N}, f_{N} )
& = \kappa^{-1}\int_{\p \wt\Omega}\Lambda_{\wt{\mathbf{C}}}\wt G_{N}\cdot \overline{\wt G_{N}}~ dS+\sum _{\alpha \in \N^2} (f_{N}|\bold e_{\alpha_1})(\overline{f_{N}}|\bold e_{\alpha_2})  X_\alpha \\
&=  \notag  \sum_{i,j=1}^{3}Z_{ij}(P)a_{i}\overline{a_{j}} + \sum _{\alpha \in \N^2}  (f_N|\bold e_{\alpha_1})(\overline{f_N}|\bold e_{\alpha_2})   X_\alpha \\ 
&\quad + O\left(N^{-1/2}+\mathcal{E}(M)\right),
\end{align}
where $Z_{ij}$ is given by \eqref{eq:Z_ij} and $\mathcal{E}(M)$ is an error term given by 
$$
\mathcal{E}(M):=N^{1/2}\left(\int_{|y'|\leq N^{-1/2}}|\nabla'\phi(y'+p')-\nabla'\phi(p')|^2~dy'\right)^{1/2}.
$$


We then state the following proposition by replacing $L^2(\p \Omega)$ by $(L^{2}(\partial\Omega))^3$ in Lemma 2.3 of \cite{caro2017calderon}. 

\begin{proposition}\label{Prop auxilliary in elast}
	Let $\mathbf{C}$ be the isotropic elastic tensor and $\Omega$ be a bounded Lipschitz domain in $\R^3$. Then
 there is a complete probability space $(\Pi,\mathcal H, \mathbb P)$, and a countable family $\left\{X_\alpha : \alpha \in \N^2\right\}$ of independent complex random variables satisfying \eqref{eqn:GRV}. In addition, for any $f,g\in (L^2(\p \Omega))^3$, we have 
		\begin{align}\label{Parseval id}
		\E\left|\sum_{\alpha\in\N^2}(f|{\bold e}_{\alpha_1})(g|{\bold e}_{\alpha_2})X_\alpha\right|^2=\|f\|^2_{(L^{2}(\partial\Omega))^3}\|g\|^2_{(L^{2}(\partial\Omega))^3}.
		\end{align}
		
Furthermore, the corrupted data 
		$$
		\mathcal{N}_{\mathbf{C}}: (H^{1/2}(\p \Omega ))^3\times (H^{1/2}(\p \Omega))^3\to L^2(\Pi, \mathcal H, \mathbb P)
		$$
		and the following estimate holds 
		\begin{align}\label{H^1/2 bound for elasticity}
		\E\left|\mathcal{N}_{\mathbf{C}}(f,g)\right|^2 \leq C\left(1+\|\lambda\|_{L^\infty(\Omega)}+\|\mu \|_{L^\infty(\Omega)}\right)\|f\|_{(H^{1/2}(\p \Omega))^3}^2 \|g\|_{(H^{1/2}(\p \Omega))^3}^2,
		\end{align} 
		for any $f,g\in (H^{1/2}(\p \Omega))^3$, and for some constant $C>0$ depending on $\p \Omega$. In particular, \eqref{H^1/2 bound for elasticity} implies that $\left|\mathcal{N}_{\mathbf{C}}(f,g)\right|<\infty $ almost surely.

\end{proposition}

From \eqref{Parseval id}, one can obtain
\begin{align}\label{L2 bound f}
\E\left|\sum_{\alpha\in\N^2}(f_N|{\bold e}_{\alpha_1})(f_N|{\bold e}_{\alpha_2})X_\alpha\right|^2=\|f_N\|^4_{(L^{2}(\partial\Omega))^3}\leq C_{\partial \Omega} N^{-2},
\end{align}
for some constant $C_{\partial \Omega}>0$ depending only on the Lipschitz function $\phi$, where the last inequality comes from the definition of the oscillating boundary data. 

Under some suitable assumptions on the boundary $\p \wt \Omega$, the last term in \eqref{noisy bilinear 2} converges to zero as $N\to \infty$. Thus, the Lam\'e parameters $\lambda$ and $\mu$ at $P\in \p \Omega$ can be reconstructed from $\mathcal{N}_{\mathbf{C}}(f_N, f_N)$ and \eqref{L2 bound f} by taking $N\to \infty$.

 	
\begin{proof}[Proof of Theorem \ref{Thm:elastic boundary determination}]
	Following the argument of \cite{caro2017calderon} or of Section \ref{sec:3}, one can obtain the boundary determination as well as the rate of convergence for Lam\'e moduli, which finishes the proof of Theorem \ref{Thm:elastic boundary determination}. 
\end{proof}



\vskip 0.5cm
\noindent\textbf{Acknowledgement.}
P. Caro was supported by MTM2018, Severo Ochoa SEV-2017-0718 and BERC 2018-2021.
R.-Y. Lai is partially supported by NSF grant DMS-1714490. Y.-H. Lin is supported by the Academy of Finland, under the project number 309963. T. Zhou is supported by NSF grant DMS-1501049.

\bibliographystyle{alpha}
\bibliography{ref}

\end{document}